\documentclass{amsart}

\RequirePackage{color}

\def\smallskip{\vskip\smallskipamount}
\def\medskip{\vskip\medskipamount}
\def\bigskip{\vskip\bigskipamount}

\usepackage{amsmath,amsthm,bm,mathrsfs,amscd,tikz}
\usepackage{ytableau}
\usepackage{comment}
\usepackage{mathdots}
\usepackage[all]{xy}
\usepackage{graphicx}
\usepackage[colorinlistoftodos]{todonotes}
\usepackage{enumerate}
\usepackage{feynmp-auto}
\usepackage[english]{babel}
\usepackage[utf8]{inputenc}
\usepackage{float}
\usepackage{tikz}
\usepackage{tikz-cd}
\usepackage{amssymb}
\usepackage{mathtools}
\usepackage{soul}
\usetikzlibrary{decorations.pathmorphing}

\usepackage[utf8]{inputenc}

\newcommand{\nocontentsline}[3]{}
\let\origcontentsline\addcontentsline
\newcommand\stoptoc{\let\addcontentsline\nocontentsline}
\newcommand\resumetoc{\let\addcontentsline\origcontentsline}

\newtheoremstyle{thmstyle}{}{}{\itshape}{}{\bfseries}{ }{5pt}{}
\newtheoremstyle{exstyle}{}{}{}{}{\bfseries}{ }{5pt}{}
\newtheoremstyle{defstyle}{}{}{}{}{\bfseries}{ }{5pt}{}
\newtheoremstyle{remstyle}{}{}{}{}{\bfseries}{ }{5pt}{}

\theoremstyle{thmstyle}
\newtheorem{thm}{Theorem}[section]
\newtheorem{theorem}[thm]{Theorem}

\newtheorem{lemma}[thm]{Lemma}

\newtheorem{prop}[thm]{Proposition}
\newtheorem{proposition}[thm]{Proposition}

\newtheorem{corollary}[thm]{Corollary}

\theoremstyle{exstyle}

\newtheorem{example}[thm]{Example}

\theoremstyle{defstyle}

\newtheorem{def-prop}[thm]{Definition-Proposition}
\newtheorem{def-lem}[thm]{Definition-Lemma}
\newtheorem{rem-convention}[thm]{Remark-Convention}
\newtheorem{def-note}[thm]{Definition-Notation}

\theoremstyle{remstyle}

\newtheorem{remark}[thm]{Remark}

\theoremstyle{remstyle}

\newcommand{\Hom}{\operatorname{Hom}}
\newcommand{\Fac}{\operatorname{Fac}}
\newcommand{\Ext}{\operatorname{Ext}}

\newcommand{\taurigid}{\operatorname{\tau-rigid}}

\DeclareMathOperator{\End}{End}
\DeclareMathOperator*{\rad}{rad}

\DeclareMathOperator*{\modu}{mod}

\DeclareMathOperator*{\proj}{proj}

\DeclareMathOperator*{\Sub}{Sub}

\DeclareMathOperator*{\Filt}{Filt}

\DeclareMathOperator*{\ind}{ind}

\DeclareMathOperator*{\brick}{brick}

\DeclareMathOperator*{\simp}{sim}
\DeclareMathOperator*{\tors}{tors}
\DeclareMathOperator*{\chtors}{ch-tors}

\DeclareMathOperator*{\torf}{torf}

\DeclareMathOperator*{\soc}{soc}
\DeclareMathOperator*{\topm}{top}

\DeclareMathOperator*{\chbrick}{ch-brick}
\DeclareMathOperator*{\JI}{JI}
\DeclareMathOperator*{\MI}{MI}
\DeclareMathOperator*{\Hasse}{Hasse}

\newcommand{\cupdot}{\mathbin{\mathaccent\cdot\cup}}

\makeatletter
\newcommand*{\doublerightarrow}[2]{\mathrel{
  \settowidth{\@tempdima}{$\scriptstyle#1$}
  \settowidth{\@tempdimb}{$\scriptstyle#2$}
  \ifdim\@tempdimb>\@tempdima \@tempdima=\@tempdimb\fi
  \mathop{\vcenter{
    \offinterlineskip\ialign{\hbox to\dimexpr\@tempdima+1em{##}\cr
    \rightarrowfill\cr\noalign{\kern.5ex}
    \rightarrowfill\cr}}}\limits^{\!#1}_{\!#2}}}
\newcommand*{\triplerightarrow}[1]{\mathrel{
  \settowidth{\@tempdima}{$\scriptstyle#1$}
  \mathop{\vcenter{
    \offinterlineskip\ialign{\hbox to\dimexpr\@tempdima+1em{##}\cr
    \rightarrowfill\cr\noalign{\kern.5ex}
    \rightarrowfill\cr\noalign{\kern.5ex}
    \rightarrowfill\cr}}}\limits^{\!#1}}}
\makeatother

\makeatletter
\newcommand{\doublewidetilde}[1]{{%
  \mathpalette\double@widetilde{#1}%
}}
\newcommand{\double@widetilde}[2]{%
  \sbox\z@{$\m@th#1\widetilde{#2}$}%
  \ht\z@=.9\ht\z@
  \widetilde{\box\z@}%
}
\makeatother

\begin{document}

\title[Brick-splitting torsion pairs $\&$ left modularity]{Brick-splitting torsion pairs and left modularity}

\author[Sota Asai, Osamu Iyama, Kaveh Mousavand, Charles Paquette]{Sota Asai, Osamu Iyama, Kaveh Mousavand, Charles Paquette} 

\address{Sota Asai: Graduate School of Mathematical Sciences, the University of Tokyo, Japan}
\email{sotaasai@g.ecc.u-tokyo.ac.jp}
\address{Osamu Iyama: Graduate School of Mathematical Sciences, the University of Tokyo, Japan}
\email{iyama@ms.u-tokyo.ac.jp}
\address{Kaveh Mousavand: Representation Theory and Algebraic Combinatorics Unit, Okinawa Institute of Science and Technology (OIST), Japan}
\email{mousavand.kaveh@gmail.com}
\address{Charles Paquette: Department of Mathematics and Computer Science, Royal Military College of Canada, Kingston ON, Canada}
\email{charles.paquette.math@gmail.com}

\subjclass [2020]{16G10, 06A07, 05E10, 16S90}
\keywords{brick, splitting torsion pair, representation-directed algebra, brick-directed algebra, trim lattice, wall-and-chamber structure}

\dedicatory{Dedicated to Claus Michael Ringel on the occasion of his 80th birthday}

\begin{abstract} 
We introduce the notion of brick-splitting torsion pairs as a modern analogue and generalization of the classical notion of splitting torsion pairs. A torsion pair is called brick-splitting if any given brick is either torsion or torsion-free with respect to that torsion pair. After giving some properties of these pairs, we fully characterize them in terms of some lattice-theoretical properties, including left modularity. This leads to the notion of brick-directed algebras, which are those for which there does not exist any cycle of non-zero non-isomorphisms between bricks. This class of algebras is a novel generalization of representation-directed algebras. We show that brick-directed algebras have many interesting properties and give several characterizations of them. In particular, we prove that a brick-finite algebra is brick-directed if and only if the lattice of torsion classes is left modular (or equivalently, extremal). We also give a characterization of brick-directed algebras in terms of their wall-and-chamber structure, as well as of a certain Newton polytope associated to them. Moreover, we introduce an explicit construction of an abundance of brick-directed algebras, both of the tame and wild representation types. 
\end{abstract}

\maketitle

\tableofcontents

\newpage

\section{Summary and Main Results}\label{Sec: Summary of Main Results}

In this section, we outline some of the important motivations of our work and summarize our main results. For terminology and notations that are not explicitly defined here, we refer to the end of this section and Section \ref{Section: Preliminaries and Background}, and references therein. 
In the following, $A$ is always assumed to be a finite dimensional associative $k$-algebra with identity, where $k$ denotes a field. By $\modu A$, we denote the category of finitely generated left $A$-modules.

\subsection{Brick-splitting torsion pairs}\label{Subsection: Brick-splitting torsion pairs}
For an algebra $A$, we recall that a torsion pair $(\mathcal{T},\mathcal{F})$ in $\modu A$ is said to be \emph{splitting} if $\ind A \subseteq \mathcal{T}\cup \mathcal{F}$, that is, every indecomposable module in $\modu A$ belongs to $\mathcal{T}$ or $\mathcal{F}$. Several characterizations of splitting torsion pairs have appeared in the literature (see Proposition \ref{Prop: splitting pairs}). 
As one of the main objectives of this paper, we introduce and study a modern analogue of splitting torsion pairs. First, recall that a module $X$ is a \emph{brick} provided $\End_A(X)$ is a division algebra. By $\brick A$ we denote the set of all (isomorphism classes) of bricks in $\modu A$. 
Then, we say that a torsion pair $(\mathcal{T},\mathcal{F})$ in $\modu A$ is \emph{brick-splitting} if $\brick A\subseteq \mathcal{T}\cup \mathcal{F}$, that is, every brick in $\modu A$ belongs to $\mathcal{T}$ or $\mathcal{F}$. A torsion class (or a torsion-free class) that belongs to a brick-splitting torsion pair is also called \emph{brick-splitting}. 
Each splitting torsion pair is obviously brick-splitting, but the converse is not true (see Example \ref{Example: Non-trivial brick-splitting torsion pairs}).  Henceforth, we refer to $(0,\modu A)$ and $(\modu A,0)$ as the \emph{trivial} (brick-)splitting torsion pairs.

\medskip

Our first theorem gives several characterizations of brick-splitting torsion pairs. Before stating these results, we recall some standard notation and terminology. In particular, for a lattice $(L,\leq)$, by $\Hasse(L)$ we denote the Hasse diagram of $L$. Moreover, an element $x$ in $L$ is called \emph{left modular} if for each $y \leq z$ in $L$, we have the equality $(y\vee x)\wedge z = y\vee (x\wedge z)$. 
For an algebra $A$, by $\tors A$ we denote the set of all torsion classes in $\modu A$. It is known that, under the inclusion order, $\tors A$ is a complete lattice which is weakly atomic semidistributive, and a labeling of the edges of $\Hasse(\tors A)$ is given by bricks. For some details on the lattice theory and properties of $\tors A$, see Section \ref{Section: Preliminaries and Background} and references therein. 

\begin{theorem}[Prop. \ref{Prop: brick-splitting and intervals} $\&$ Prop. \ref{Prop: brick-splitting is left modular}]\label{Thm: brick-splitting torsion class characterization}
For an algebra $A$, the following are equivalent:
\begin{enumerate}
    \item $(\mathcal{T}, \mathcal{F})$ is a brick-splitting torsion pair.
    \item Every $B$ in $\brick A$ appears as an arrow in $\Hasse [0,\mathcal{T}]$ or $\Hasse[\mathcal{T}, \modu A]$.
    \item $\mathcal{T}$ is a left modular element of the lattice $\tors A$.
\end{enumerate}
\end{theorem}

The next proposition provides further insights into the behavior of brick-splitting torsion pairs and shows how such pairs control some important properties of the algebras under consideration. 
For a full subcategory $\mathcal{C}$ of $\modu A$, by $\simp A\cap \mathcal{C}$ we denote the set of isomorphism classes of simple $A$-modules that belong to $\mathcal{C}$. Note that for each torsion pair $(\mathcal{T}, \mathcal{F})$ in $\modu A$, we always have $\simp A \subseteq \mathcal{T} \cup\mathcal{F}$. 

\begin{proposition}[Prop. \ref{Prop: extension between simples}] \label{Prop:brick-splitting and simples}
Let $A$ be a connected algebra and $(\mathcal{T},\mathcal{F})$ be a brick-splitting torsion pair in $\modu A$. Then, for $S_1 \in \simp A\cap \mathcal{T}$ and $S_2 \in \simp A \cap \mathcal{F}$, we have $\Ext^1_A(S_2, S_1)=0$. Hence, in the Ext-quiver of $A$, there is no arrow from the vertices corresponding to $\simp A \cap \mathcal{F}$ to the vertices corresponding to $\simp A \cap \mathcal{T}$.
\end{proposition}

Let $A$ and $(\mathcal{T},\mathcal{F})$ be as in Proposition \ref{Prop:brick-splitting and simples}, and by $e_{\mathcal{T}}$ denote the idempotent of $A$ corresponding to the sum of the primitive idempotents corresponding to $\simp A \cap \mathcal{T}$. Then, $e_{\mathcal{T}}A(1-e_{\mathcal{T}})=0$. Consequently, $A$ admits a triangular decomposition
$$A \cong \left[\begin{array}{cc}
   e_{\mathcal{T}}Ae_{\mathcal{T}}  &  0\\
   (1-e_{\mathcal{T}})Ae_{\mathcal{T}}  & (1-e_{\mathcal{T}})A(1-e_{\mathcal{T}})\\ 
\end{array}\right].$$
Motivated by this observation, we say an algebra $A$ is \emph{fully cyclic} provided that for any idempotent $e$ of $A$, the condition $eA(1-e)=0$ implies $e=0$ or $e=1$. The following corollary is an interesting consequence of Proposition \ref{Prop:brick-splitting and simples}.

\begin{corollary}[Cor. \ref{Cor: fully cyclic and self-injective}]
\label{Cor: No brick-splitting over self-injective algebras}
If $A$ is fully cyclic, then the only brick-splitting torsion pairs in $\modu A$ are trivial. In particular, if $A$ is a connected self-injective algebra, then the only brick-splitting torsion pairs in $\modu A$ are $(0,\modu A)$ and $(\modu A, 0)$.
\end{corollary}

As our second novel generalization of a classical notion in representation theory, we give a brick-analogue of the family of representation-directed algebras (for the definition and some remarks on representation-directed algebras, see Section \ref{Section: Brick-directed algebras}). 
For an algebra $A$, we say $X_0\xrightarrow{f_1} X_1\xrightarrow{f_2}  \cdots \xrightarrow{f_{m-1}}  X_{m-1}\xrightarrow{f_m}  X_m$ is a \emph{brick-cycle} in $\modu A$ if every $X_i$ is a brick, for all $0\leq i\leq m$, with $X_0\simeq X_m$, and each $f_j:X_{j-1}\rightarrow X_j$ is non-zero and non-invertible, for $1\leq j \leq m$. Then, we say $A$ is a \emph{brick-directed} algebra if $\modu A$ contains no brick-cycles. Evidently, each representation-directed algebra is brick-directed. However, there are many brick-directed algebras that are not representation-directed  (see Examples \ref{Example: Non-trivial brick-splitting torsion pairs}, \ref{Example: windwheel alg of rank 2}, and Corollary \ref{Various families of brick-directed algebras}).

\medskip

In the following theorem, we give a conceptual characterization of brick-directed algebras in terms of brick-splitting torsion pairs. 
Let us recall that, for a poset $(P,\leq)$, a subset of $P$ is called a \emph{chain} if it is totally ordered. A chain in $P$ is said to be \emph{maximal} if it is not properly contained in any other chain in $P$.

\begin{theorem}[Theorem \ref{Thm: Brick-directed & maximal chain of brick-splitting pairs}]\label{Thm: Characterization of brick-directed via maximal chain of brick-splitting pairs}
For an algebra $A$, the following are equivalent:
\begin{enumerate}
    \item $A$ is brick-directed;
    \item There exists a maximal chain $\{\mathcal{T}_i\}_{i\in I}$ in $\tors A$ such that every $\mathcal{T}_i$ is a brick-splitting torsion class.
\end{enumerate}
\end{theorem}

Similar to Proposition \ref{Prop:brick-splitting and simples}, we have the following result on the Ext-quiver of brick-directed algebras. In particular, the following corollary is a consequence of Proposition \ref{prop: chain of brick-splitting separating simples}. 
We say that an algebra $A$ is \emph{weakly triangular} if the only oriented cycles in the Ext-quiver of $A$ are powers and compositions of loops. 
 
\begin{corollary}[Cor. \ref{Cor: brick-directed is weak-triangular}]\label{Cor: No Cycle}
Every brick-directed algebra is weakly triangular.
\end{corollary} 

As noted above, through brick-directed algebras, we obtain a vast generalization of the classical family of representation-directed algebras. Unlike representation-directed algebras, which are always representation-finite, brick-directed algebras include various types of representation-infinite algebra. In fact, via a simple gluing construction recalled in Section \ref{Section: Brick-directed algebras}, one obtains an abundance of explicit brick-directed algebras of any given rank $n \in \mathbb{Z}_{>1}$. 
Before we state the next result, recall that by Drozd's trichotomy theorem, each algebra $A$ over an algebraically closed field $k$ falls exactly into one of the families of representation-finite, (representation-infinite) tame or wild algebras.

\begin{corollary}[Cor. \ref{Various families of brick-directed algebras}]\label{Cor on various families of brick-directed algebras}
For any positive integer $n>1$, there exists a brick-directed algebra of rank $n$ of any of the following types:
\begin{itemize}
    \item representation-finite algebra;
    \item brick-finite tame algebra;
    \item brick-infinite tame algebra;
    \item brick-finite wild algebra;
    \item brick-infinite wild algebra.
\end{itemize}
\end{corollary}

Notice that, by considering a particular class of bricks, a similar (but weaker) notion to brick-directed algebras (algebras admitting a \emph{reverse hom-orthogonal order}) was studied in \cite[Definition 5.5]{BH}. We observe that for brick-finite algebras, the aforementioned notion coincides with the notion of brick-directedness that we introduced above.

\subsection{$\tau$-tilting finite algebras and left modular lattices} 
In Section \ref{Section: Trim lattices in representation theory}, we use brick-directed algebras to give new characterizations of some lattice theoretical phenomena in the context of lattice theory of torsion classes. 
Before stating some of our results, we remark that in the literature of lattice theory many important combinatorial properties of finite lattices are not defined for infinite lattices. Consequently, we restrict our attention to those algebras $A$ for which $\tors A$ is a finite lattice. This is the case if and only if $A$ is $\tau$-tilting finite, which itself is equivalent to brick-finiteness of $A$ (see \cite{DIJ, DI+}). Hence, in this part we only consider brick-finite algebras.

\medskip
We first recall that for any algebra $A$, the lattice $\tors A$ is always semidistributive (see \cite[Theorem 1.3]{DI+} and \cite[Corollary 8.8]{RST}), a weaker version of distributive lattices.
We also recall that, in lattice theory, there are some other weaker versions of distributive lattices.
Firstly, $(L,\leq)$ is called \emph{left modular} if $L$ has a maximal chain of left modular elements. Secondly, $(L,\leq)$ is said to be \emph{extremal} if it has a maximal chain whose length is equal to both the number of join-irreducible elements and the number of meet-irreducibles \cite{Ma}. Thirdly, a finite lattice $(L,\leq)$ is said to be \emph{trim} if it is extremal and left modular \cite{Th}. 
Through some lattice-theoretical techniques, it has recently been shown that for semidistributive lattices, left modularity, extremality, and trimness are equivalent (see \cite{TW} and \cite{Mu}).
However, without semidistributivity, neither of left modularity and extremality implies the other (see Section \ref{Section: Preliminaries and Background} and references therein).

\medskip

The following theorem gives 
new realizations of the above-mentioned lattice-theoretical properties in the more algebraic framework of representation theory (see Propositions \ref{Prop: brick-directed equiv. left modular} and \ref{Prop: brick-directed equiv. extremal}). In particular, we give a complete classification of those algebras $A$ for which $\tors A$ are left modular (equivalently, extremal). 
Before we summarize our results in the next theorem, recall that for an extremal lattice $L$, the \emph{spine} of $L$ consists of those elements that lie on some chain of maximal length.

\begin{theorem}[Prop. \ref{Prop: brick-directed equiv. left modular} $\&$ Prop. \ref{Prop: brick-directed equiv. extremal} $\&$ Cor. \ref{Cor: brick-directed equiv. trim}]\label{Thm: brick-directed, extremal, left modular, trimness}
Let $A$ be a brick-finite algebra. The following are equivalent:
\begin{enumerate}
    \item $A$ is brick-directed;
    \item $\tors A$ is left modular;
    \item $\tors A$ is extremal;
    \item $\tors A$ is a trim lattice.
\end{enumerate}
Moreover, if $A$ satisfies any of the above conditions, then the spine of  $\tors A$ consists of all brick-splitting torsion classes and it forms a distributive sublattice. 
\end{theorem}

We note that the final assertion of the preceding theorem (i.e., the characterization of the spine), is a consequence of Theorem \ref{Thm: brick-splitting torsion class characterization}, together with some earlier results on trim lattices (see \cite[Prop. 2.6 and Theorem 3.7]{TW}). We observe that from the aforementioned characterization of spine of trim lattices in our setting, one concludes that for a brick-finite brick-directed algebra $A$, each brick-splitting torsion class in $\modu A$ lies on a chain of maximal length in $\tors A$, that is, a chain of length $|{\brick A}|$. From Theorem \ref{Thm: brick-directed, extremal, left modular, trimness}, we also derive some interesting results. 

\begin{corollary}[Cor. \ref{Cor: quotien, corner and tau-reduction of extremal lattices}]\label{Cor: extremailty preserved}
Let $A$ be a brick-finite algebra. If $\tors A$ is extremal  (equivalently, left modular), then the following lattices are also extremal:
\begin{enumerate}
    \item $\tors A/J$, for any 2-sided ideal $J$ in $A$;
    \item $\tors eAe$, for any idempotent element $e$ in $A$;
    \item $\tors B$, for any algebra $B$ which is a $\tau$-reduction of $A$.
\end{enumerate}
\end{corollary}

Although some earlier studies of trim lattices in the context of lattice theory of torsion classes have already appeared in the literature, such results are primarily in the setting of representation-finite algebras. In \cite[Corollary 1.5]{TW} the authors showed that if $A$ is representation-directed, then $\tors A$ is a trim lattice. In fact, Theorem \ref{Thm: brick-directed, extremal, left modular, trimness} strengthens and greatly generalizes such former results and also leads to a lattice theoretical characterization of representation-directed algebras. 
Recall that for a finite lattice $(L,\leq)$, \emph{length} of $L$ is defined to be the length of a longest maximal chain in $L$.

\begin{corollary}[Cor. \ref{Cor: charact. of rep-directed via brick-directed}]\label{Cor: Characterization of rep-directed algebras}
Let $A$ be a brick-finite algebra. Then, $A$ is representation-directed if and only if $A$ is representation-finite and length of $\tors A$ is $|{\ind A}|$.
\end{corollary}

In Section \ref{Subsection: Wall-chamber structure of brick-directed algebras}, we turn our attention to the characterization of brick-directed algebras via their wall-and-chamber structures. After a brief recollection of some standard tools from that setting, in the aforementioned section we introduce the notion of a (weakly) consistent sequence of elements in the Grothendieck group $K_0(\proj A)_{\mathbb{R}}$, and prove the following result. For the undefined terminology, we refer to Section \ref{Subsection: Wall-chamber structure of brick-directed algebras}.

 \begin{theorem}[Theorem \ref{Prop: walls ordered consistently}]\label{Thm: wall-chamber}
     Let $A$ be a brick-finite algebra. The following are equivalent:
\begin{enumerate}
    \item $A$ is brick-directed;
    \item $K_0(\proj A)_{\mathbb{R}}$ admits a consistent sequence;
    \item $K_0(\proj A)_{\mathbb{R}}$ admits a weakly consistent sequence.
\end{enumerate}
 \end{theorem}

Our next result states a remarkable property of bricks over brick-directed algebras. More precisely, from Theorems \ref{Thm: brick-splitting torsion class characterization} and \ref{Thm: brick-directed, extremal, left modular, trimness}, and some further observations, in Section \ref{Section: Trim lattices in representation theory} we show the following uniqueness property that generalizes some classical results. We particularly note that every representation-directed algebra satisfies the assumptions of the following statement.

\begin{corollary}[Cor. \ref{Cor: unique dim vector}]\label{Cor: Uniqueness of dim vector of bricks}
Let $A$ be a brick-finite algebra. If $A$ is brick-directed, then each brick in $\modu A$ is uniquely determined by its dimension vector, that is, for two distinct elements $X$ and $Y$ in $\brick A$, we have $\operatorname{\underline{dim}} X\neq \operatorname{\underline{dim}} Y$.
\end{corollary}

We remark that the converse of the previous corollary is not true in general. Meanwhile, let us highlight two interesting consequences of the preceding corollary, which also motivate some related properties that can be studied for arbitrary brick-directed algebras (for details, see Remark \ref{Rem: Unique dim vect but not brick-direcetd}). 
Working with a brick-finite algebra, say $A$,  Corollary \ref{Cor: Uniqueness of dim vector of bricks} implies the following.
On the one hand, it provides a simple necessary condition for the brick-directedness: If two distinct bricks have the same dimension vector, there exists a cycle of bricks, hence $A$ is not brick-directed. 
On the other hand, this corollary implies a strong condition on the module varieties of brick-directed algebras: If $A$ is brick-directed, for each dimension vector $\underline{d} \in \mathbb{Z}^n_{\geq 0}$, the variety $\modu(A,\underline{d})$ contains at most one brick component. 

\medskip

The last main result that we present here gives another characterization of brick-directed algebras in terms of a certain Newton polytope naturally associated to any brick-finite algebra. For undefined terminology and notation, we refer to Section \ref{Section: Trim lattices in representation theory}.

\begin{theorem}[Theorem \ref{Thm: Newton polytope}]\label{Newton polytope thm in Introduction}
Let $A$ be brick-finite, and set $M$ as the direct sum of $X \in \brick A$. Then, $A$ is brick directed if and only if 
there exists an indivisible increasing path in $\mathrm{N}(M)$ from $0$ to $[M]$.
\end{theorem}

We end this section with a remark that puts our work in a broader perspective.

\begin{remark}
As observed above, using brick-finite brick-directed algebras, we gave a full characterization of trim lattices in the context of lattice of torsion classes. Meanwhile, we emphasize that some of our main results on brick-directed algebras are independent of the finiteness assumption on the lattice $\tors A$. In particular, there exist many examples of brick-directed algebras for which $\tors A$ is an infinite lattice (see Corollary \ref{Cor on various families of brick-directed algebras}). 
Hence, in the context of lattice theory of torsion classes, the notion of brick-directed algebras can be seen as a generalization of trimness from the setting of finite semidistributive lattices to that of infinite semidistributive lattices. 
Thus, our work can potentially pave the way for the study of some fundamental lattice-theoretical phenomena. More specifically, thanks to the recent results of \cite{TW} and \cite{Mu}, it is known that for any finite semidistributive lattice $(L,\leq)$, the notions of extremality and left modularity are equivalent (see \cite[Corollary 3.5]{Mu}). Hence, in light of Theorems \ref{Thm: Characterization of brick-directed via maximal chain of brick-splitting pairs} and \ref{Thm: brick-directed, extremal, left modular, trimness}, it is natural to investigate the notions of extremality and left modularity over infinite semidistributive lattices.
\end{remark}

\medskip

\noindent \textbf{Notations and Setting.}\label{Notation & Setting}
Throughout, $k$ denotes a field, and $A$ is always assumed to be a finite dimensional associative $k$-algebra with multiplicative identity. Without loss of generality, $A$ is always assumed to be basic. By $\{e_1,\ldots, e_n\}$ we denote a complete set of primitive orthogonal idempotents, and we let $\{S_1, \ldots, S_n\}$ be a complete set of non-isomorphic simple $A$-modules where $e_iS_j \ne 0$ if and only if $i=j$, for $1\leq i, j \leq n$. Moreover, for $i=1,2,\ldots,n$, we denote by $P_i = Ae_i$ the projective cover of $S_i$.

By $\Omega = \Omega_A$ we denote the (unvalued version of the) \emph{Ext-quiver} of algebra $A$ whose vertex set is given by $\{S_1, \ldots, S_n\}$ and there exists a single arrow $S_i \to S_j$ if and only if $\Ext^1_A(S_i, S_j) \ne 0$. When $A$ is elementary, that is when all $S_i$ are one dimensional over $k$, we can also associate to $A$ its \emph{ordinary quiver} $Q =Q_A$. It has the same vertex set as $\Omega$ and there are $\dim_k\Ext^1_A(S_i, S_j)$ many arrows from $S_i$ to $S_j$ in $Q$. In such a case, $A \cong kQ/I$ for some admissible ideal $I$. Note that when $k$ is algebraically closed and $A$ is basic, then $A$ is necessarily elementary. We observe the following fact about the Ext-quiver of $A$. If there is no path from $S_i$ to $S_j$ in $\Omega$, then $\Hom_A(Ae_j, Ae_i)=0$. In particular, the algebra $A$ is connected if and only if $\Omega$ is connected.
For more details on the Ext-quiver, see \cite[III.1]{ARS}.

By $\modu A$ we denote the category of all finitely generated left $A$-modules. If $A \cong kQ/I$, then each $X$ in $\modu A$ can be seen as a finite dimensional representation of $(Q,I)$. Unless specified otherwise, modules are always considered up to isomorphism. Let $\ind A$ denote the set of all isomorphism classes of indecomposable objects in $\modu A$, and let $\Gamma_A$ denote the Auslander-Reiten quiver of $\modu A$. For all rudiments of representation theory, we refer to \cite{ARS} and \cite{ASS}.

\section{Preliminaries and Background} \label{Section: Preliminaries and Background}

\stoptoc

\subsection{Lattice theory}\label{Subsection: Lattice theory}
In this subsection, we recall only some facts from lattice theory that are used in our work. For the most part, we only provide references. For all the rudiments of lattice theory, we refer to \cite{Gr} and \cite{DP}.

Let $(P,\leq)$ be a non-empty (possibly infinite) partially ordered set. Recall that $x<z$ in $P$ is said to be a \emph{cover} relation, and it is denoted by $x\lessdot z$, provided that for each $y\in P$, if $x\leq y \leq z$, we have either $x=y$ or $y=z$. Furthermore, the \emph{Hasse quiver} of $P$, denoted by $\Hasse(P)$, is a directed graph which has $P$ as the vertex set, and arrows $x\rightarrow z$ for each cover relation $x\lessdot z$. 

We also recall that $(P,\leq)$ is said to be a \emph{complete lattice} if for any subset $S$ in $P$, there exists a unique element of $P$, smallest with the property of being larger than or equal to all elements of $S$, called the \emph{join} of $S$ and denoted by $\bigvee S$; as well as a unique element of $P$, largest with the property of being smaller than or equal to all elements of $S$, called the \emph{meet} of $S$ and denoted by $\bigwedge S$. We often use $(L,\leq)$ to denote a lattice, and for $x$ and $y$ in a lattice $L$, the join and meet are respectively denoted by $x\vee y$ and $x\wedge y$. 

In the following, $(L,\leq)$ always denotes a complete lattice, and $1$ and $0$ respectively denote the greatest element and the least element of $L$ (also known as the maximum and minimum of $L$). Whenever there is no confusion, we suppress $\leq$ in the notation and simply use $L$ instead of $(L,\leq)$. Recall that a lattice $L$ is said to be \emph{weakly atomic} if for each $x<y$ in $L$, there is at least one arrow in $\Hasse[x,y]$.

A sequence $x_0<x_1<\cdots<x_i<\cdots$ in $L$ is called a \emph{chain}, and is said to be finite if the sequence consists of only finite many elements. In particular, the \emph{length} of $x_0<x_1<\cdots<x_r$ is defined to be $r$. Moreover, the length of the lattice $L$ is defined as the maximum of the lengths of finite chains, if there exists such a maximum. Otherwise, $L$ is said to be of infinite length. A chain $x_0<x_1<\cdots<x_i<\cdots$ is called \emph{maximal} if it cannot be further refined to a chain by inserting more elements, which is the case if and only if $x_j\lessdot x_{j+1}$, for each $j$.

An element $x \in L$ is called \emph{join-irreducible} if, for each $y$ and $z$ in $L$ with $y, z <x$, we have $x\neq y\vee z$. By convention, $0$ is not considered as a join-irreducible element. 
Provided $L$ is a finite lattice, $x$ is join-irreducible if and only if
it covers exactly one element of $L$. Moreover, $x \in L$ is said to be \emph{completely join-irreducible} if for any subset $S\subseteq \{y \in L \mid y<x\}$, we have $x \neq \bigvee_{y\in S} y$. By $\JI(L)$ we denote the set of all join-irreducible elements in $L$, and $\JI^c(L)$ denotes the subset of $\JI(L)$ consisting of completely join-irreducible elements.
The notion of (completely) meet-irreducible elements, and consequently the sets $\MI(L)$ and $\MI^c(L)$, are defined dually. Note that, by convention, $1$ does not belong to $\MI(L)$. 

A lattice $L$ is called \emph{semidistributive} if, for all $x, y, z \in L$, if $x\vee y = x\vee z$, then $ x \vee (y \wedge z)=x \vee y$, and, furthermore, $x\wedge y = x\wedge z$ implies $x \wedge (y \vee z)=x \wedge y$. 
Consequently, for any cover relation $x\lessdot y$, the set $\{a \in L \,|\, x\vee a=y$\} has a unique minimal element.
For a lattice $L$, and a pair of elements $y\leq z$ in $L$, it is easy to see that the inequality $(y\vee x)\wedge z \geq y \vee (x\wedge z)$ holds for all $x\in L$. This is known as \emph{modular inequality}. 
Motivated by this, an element $x$ in $L$ is called \emph{left modular} if for each $y < z$ in $L$, we have $(y\vee x)\wedge z = y\vee (x\wedge z)$. For more details on the left modular elements of lattices, see \cite{LS}. 
As shown in \cite[Lemma 2.13]{TW}, under the assumption that $L$ is weakly atomic, to verify the left modularity for $x \in L$, one only needs to verify $(y\vee x)\wedge z \le y \vee (x\wedge z)$ for the cover relations $y\lessdot z$ in $L$. 

\medskip

In the rest of this subsection, we restrict ourselves to the setting of finite lattices, particularly to recall some important generalizations of distributive lattices that play crucial roles in our studies. We note that, according to the current literature, many of these notions have been considered only for finite lattices. 

Let $L$ be a finite lattice of length $m$. Observe that $|{\JI(L)}|\geq m$ and $|{\MI(L)}|\geq m$, because $L$ must have at least $m$ join-irreducible elements and at least $m$ meet-irreducible elements. In fact, $L$ is known to be distributive if and only if every maximal chain in $L$ is of length $m=|{\JI(L)}|=|{\MI(L)}|$ (see \cite[Theorem 17]{Ma}). 

As a generalization of distributive lattices, following \cite{Ma}, a finite lattice of length $m$ is called \emph{extremal} if $|{\JI(L)}|=|{\MI(L)}|=m$. In particular, unlike distributive lattices, an extremal lattice may admit a maximal chain of length less than $m$ (see Example \ref{Example: Non-trivial brick-splitting torsion pairs}). For an extremal lattice $L$, \emph{spine} of $L$ consists of all elements of $L$ that lie on a chain of maximal length. The spine of an extremal lattice $L$ is known to be a distributive sublattice of $L$ (see \cite[Prop. 2.6]{TW}). 

Another generalization of distributive lattices is given in terms of left modular elements. A finite lattice $L$ is called \emph{left modular} if it has a maximal chain of left modular elements. Such lattices, at least in this terminology, seem to have appeared first in \cite{BS}. As shown later, left modular lattices can be characterized in terms of certain labelings (for details, see \cite[Section 2.5]{TW} and references therein).

Since the extremal and left modular lattices give two interesting generalizations of the finite distributive lattices, it is natural to ask about the intersection of these two families. Following \cite{Th}, we say $L$ is a \emph{trim} lattice if $L$ is simultaneously extremal and left modular. The family of trim lattices properly contains that of finite distributive lattices, yet they manifest some important properties similar to those of distributive lattices. In particular, $L$ is distributive if and only if $L$ is a graded trim lattice (\cite[Theorem 2]{Th}). 
As shown more recently in \cite{TW}, trim lattices admit many interesting combinatorial properties. 

Now we recall some important results of \cite{TW} and \cite{Mu} that relate to our work.

\begin{theorem}[{\cite[Theorem 1.4]{TW}} {\cite[Theorem 3.2]{Mu}}]\label{Thm: extremal & semidistributive is trim}
Let $(L,\leq)$ be a finite semidistributive lattice. 
Then the following are equivalent:
\begin{enumerate}
    \item $L$ is left modular;
    \item $L$ is extremal;
    \item $L$ is a trim lattice.
\end{enumerate}\end{theorem}

Using a more algebraic approach, and via our new notion of brick-splitting torsion pairs, we give a new proof and a novel realization of the above equivalences in the context of lattice of torsion classes of a brick-finite algebras (for details, see Section \ref{Section: Trim lattices in representation theory}). 
Meanwhile, for a more detailed treatment of the notions of semidistributivity, extremality, and left modularity of (finite) lattices, and their relationship with some other lattice-theoretical properties, we refer to \cite[Figure 2]{Mu}.

We note that, without semidistributvity,  the notions of extremality and left modularity do not imply each other. More explicitly, there exist extremal lattices that are not left modular, and there also exist left modular lattices that are not extremal (see \cite[Figure 5]{TW}). Moreover, as shown in \cite[Figure 3]{TW}, there exist trim lattices that are not semidistributive, and there also exist semidistributive lattices that are not trim.

\subsection{Torsion theories} Throughout, $A$ always denotes a basic connected $k$-algebra satisfying the assumptions from Section \ref{Sec: Summary of Main Results}. In the following, by a subcategory $\mathcal{C}$ of $\modu A$, we always assume that $\mathcal{C}$ is full and closed under isomorphism classes. Moreover, we set $\mathcal{C}^\perp:=\{X \in \modu A\,|\, \Hom_A(C,X)=0, \text{ for every } C \in \mathcal{C}\}$, and ${}^\perp \mathcal{C}$ is defined dually. 
Recall that a subcategory $\mathcal{C}$ is said to be \emph{extension-closed} if for each short exact sequence $0\rightarrow L \rightarrow M \rightarrow N \rightarrow 0$ in $\modu A$, if $L$ and $N$ belong to $\mathcal{C}$, then so does $M$. 

If $\mathcal{T}$ and $\mathcal{F}$ are two subcategories of $\modu A$, then $(\mathcal{T},\mathcal{F})$ is called a \emph{torsion pair} in $\modu A$ provided the following conditions hold:
\begin{enumerate}
    \item $\mathcal{T}$ and $\mathcal{F}$ have only $0$ in common.
    \item $\mathcal{T}$ is closed under taking factors and $\mathcal{F}$ is closed under taking submodules.
    \item For each $M$ in $\modu A$, there exists a unique submodule $t(M)$ of $M$ such that $0 \rightarrow t(M) \rightarrow M \rightarrow M/t(M)\rightarrow 0$ is exact with $t(M) \in \mathcal{T}$ and $M/t(M) \in \mathcal{F}$.
\end{enumerate}

If $(\mathcal{T},\mathcal{F})$ is a torsion pair in $\modu A$, the subcategories $\mathcal{T}$ and $\mathcal{F}$ are respectively called the \emph{torsion class} and the \emph{torsion-free class}. Consequently, the modules in $\mathcal{T}$ are called torsion modules and those belonging to $\mathcal{F}$ are known as torsion-free modules. Note that  both $\mathcal{T}$ and $\mathcal{F}$ are extension-closed, and each one uniquely determines the pair, because $\mathcal{F}=\mathcal{T}^\perp$ and $\mathcal{T}= {}^\perp\mathcal{F}$ (for more details, see \cite[Chapter VI.1]{ASS}). 
Since $(0,\modu A)$ and $(\modu A, 0)$ trivially form torsion pairs, we say $(\mathcal{T},\mathcal{F})$ is a non-trivial torsion pair in $\modu A$  provided $\mathcal{T}\neq 0$ and $\mathcal{T}\neq \modu A$. 

For a set of modules in $\modu A$, the following lemma gives an explicit description of the smallest torsion class in $\tors A$ containing them. Before we state this fact, let us remark that, for a subcategory $\mathcal{C}$ of $\modu A$, the smallest extension-closed subcategory of $\modu A$ that contains $\mathcal{C}$ consists of all modules in $\modu A$ which have a filtration by the objects in $\mathcal{C}$. Henceforth, we denote this category by $\Filt\mathcal{C}$. Moreover, for a subcategory $\mathcal{C}$ of $\modu A$, let $\Fac\mathcal{C}$ denote the full subcategory of $\modu A$ consisting of all $N$ in $\modu A$ which are quotients of a finite direct sum of objects in $\mathcal{C}$.

For a subcategory $\mathcal{C}$ of $\modu A$, we set 
$$\mathcal{T(\mathcal{C}}):=\Filt(\Fac\mathcal{C})\ \mbox{ and }\ 
\mathcal{F(\mathcal{C}}):=\Filt(\Sub\mathcal{C}).$$
 Thus each $M$ in $\mathcal{T}(\mathcal{C})$ has a filtration $0= M_0 \subseteq M_1 \subseteq \cdots \subseteq M_{d-1} \subseteq M_d=M$ such that for every $1 \leq i \leq d$, there exists an epimorphism $\psi_i:C_i\twoheadrightarrow M_{i}/M_{i-1}$ for some $C_i$ which is a finite direct sum of objects from $\mathcal{C}$. The following lemma is well known; for instance see \cite[Lemma 3.1]{MS}. This lemma particularly implies that each torsion (and similarly torsion-free) class is uniquely determined by the indecomposable modules it contains.
 
\begin{lemma}\label{Smallest Torsion Class}
For a subcategory $\mathcal{C}$ of $\modu A$, the smallest torsion class in $\modu A$ containing $\mathcal{C}$ is $\mathcal{T}(\mathcal{C})$, and the smallest torsion-free class containing $\mathcal{C}$ is $\mathcal{F}(\mathcal{C})$.
\end{lemma}

From the axioms of torsion theory, it is easy to see that for each torsion pair $(\mathcal{T},\mathcal{F})$ in $\modu A$, every simple $A$-module belongs to either $\mathcal{T}$ or $\mathcal{F}$. That is, each torsion pair divides the set $\{S_1, \cdots, S_n\}$ of all simple modules in $\modu A$ into two disjoint subsets.
As recalled in Section \ref{Sec: Summary of Main Results}, $(\mathcal{T},\mathcal{F})$ is called a \emph{splitting} torsion pair if for any $X \in \ind A$ we have $X \in \mathcal{T}$ or $X \in \mathcal{F}$.  
In general, an algebra $A$ may admit many (possibly infinitely many) non-trivial splitting torsion pairs. The following proposition lists some of the well-known characterizations of splitting torsion pairs.

\begin{proposition}[{\cite[VI. Prop. 1.7]{ASS}}]\label{Prop: splitting pairs}
For a torsion pair $(\mathcal{T},\mathcal{F})$ in $\modu A$, the following are equivalent:
\begin{enumerate}
    \item $(\mathcal{T},\mathcal{F})$ is splitting.
    \item The canonical short exact sequence $0 \rightarrow t(M) \rightarrow M \rightarrow M/t(M)\rightarrow 0$ splits, for each $M$ in $\modu A$.
    \item $\Ext^1_A(N,M)=0$, for all $M \in \mathcal{T}$ and $N\in \mathcal{F}$.
    \item $\tau^{-1}M \in \mathcal{T}$, for each $M\in \mathcal{T}$.
    \item $\tau N \in \mathcal{F}$, for each $N\in \mathcal{F}$.
\end{enumerate}
\end{proposition}

\subsection{Lattice of torsion classes}
For an algebra $A$, by $\tors A$ and $\torf A$, we respectively denote the set of all torsion classes and torsion-free classes in $\modu A$. These two sets actually form complete lattices with respect to the inclusion. More precisely, for a family of torsion classes $\{\mathcal{T}_i\}_{i \in I}$ in $ \tors A$ (respectively, $\{\mathcal{F}_i\}_{i \in I}$ in $\torf A$), the meet is given by $\bigwedge_{i\in I}\mathcal{T}_i = \bigcap_{i \in I}\mathcal{T}_i$ (respectively, $\bigwedge_{i\in I}\mathcal{F}_i = \bigcap_{i \in I}\mathcal{F}_i$). Moreover, the join $\bigvee_{i\in I}\mathcal{T}_i$ is defined as the intersection of all torsion classes $\mathcal{T}\in \tors A$ that contain $\bigcup_{i\in I}\mathcal{T}_i$. One analogously defines $\bigvee_{i\in I}\mathcal{F}_i$.
Furthermore, there is an anti-isomorphism of lattices between $\tors A$ and $\torf A$. More precisely, $\mathcal{T}\in \tors A$ is sent to $\mathcal{T}^{\bot} \in \torf A$, and in the opposite direction $\mathcal{F}$ is sent to $^{\bot}\mathcal{F}$. 

It is known that $\tors A$, and thus $\torf A$, is always a semidistributive lattice (\cite[Theorem 4.5]{GM}), and they are weakly atomic (\cite[Theorem 3.1]{DI+}).
Moreover, as proved in \cite{BCZ} and \cite{DI+}, there is a brick labeling of $\Hasse(\tors A)$, that is, each arrow in the Hasse quiver of $\tors A$ is uniquely labeled by an element of $\brick A$. However, we remark that a single element of $\brick A$ may appear as the label of more than one arrow in $\Hasse(\tors A)$. 
To briefly recall the idea of this labeling, we adopt the notation of \cite{DI+} and, for $\mathcal{U} \subseteq \mathcal{T}$ in $\tors A$, define 
$$\brick[\mathcal{U},\mathcal{T}]:=\{X\in \brick A \,|\, X \in \mathcal{T} \cap \mathcal{U}^{\bot}\}.$$
Then, there is an arrow $\mathcal{T} \rightarrow \mathcal{U}$ in $\Hasse(\tors A)$ if and only if $\brick[\mathcal{U},\mathcal{T}]$ has a unique element. We call this unique brick the \emph{label} of the arrow $\mathcal{T} \rightarrow \mathcal{U}$. 

We have the following important theorem. For the proof and more details, we refer to \cite[Section 3.2]{DI+}. 

\begin{theorem}[{\cite[Theorem 3.4]{DI+}}]\label{Thm: bijection between Join-irr and brick-labels}
Let $A$ be an algebra and let $\mathcal{U} \subseteq \mathcal{T}$ in $\tors A$.
\begin{enumerate}
    \item There is a bijection $\psi: \JI^c([\mathcal{U},\mathcal{T}])\rightarrow \brick[\mathcal{U},\mathcal{T}]$. In particular, $\psi$ sends $\mathcal{T}'$ to the brick that labels the unique arrow of $\Hasse[\mathcal{U},\mathcal{T}]$ that starts at $\mathcal{T}'$.
    \item There is a bijection $\phi: \MI^c([\mathcal{U},\mathcal{T}])\rightarrow \brick[\mathcal{U},\mathcal{T}]$. In particular, $\phi$ sends $\mathcal{U}'$ to the brick that labels the unique arrow of $\Hasse[\mathcal{U},\mathcal{T}]$ that ends at $\mathcal{U}'$.
\end{enumerate}
Consequently, the three sets $\brick A$, $\JI^c(\tors A)$, and $\MI^c(\tors A)$ are in bijection.
\end{theorem}

The following property is well-known and immediately follows from the definition of the labels of arrows in $\Hasse(\tors A)$.

\begin{proposition}\label{Prop: labels distinct in a path}
If $\mathcal{T}_l \to \mathcal{T}_{l-1} \to \cdots \to \mathcal{T}_0$ is a path in $\Hasse(\tors A)$ and $X_i$ is the label of $\mathcal{T}_i \to \mathcal{T}_{i-1}$ for each $i \in \{1,\ldots,l\}$, then we have $\Hom_A(X_i,X_j)=0$ for any $i<j$.    
\end{proposition}

\begin{proof}
Since $i<j$, we have $X_i \in \mathcal{T}_i \subseteq \mathcal{T}_{j-1}$ and $X_j \in \mathcal{T}_{j-1}^\perp$.
\end{proof}

\subsection{Chain of bricks}\label{Subsection: Chain of bricks}
As before, $\brick A$ denotes the set of all isomorphism classes of bricks in $\modu A$. 
For a totally ordered set $(I,\leq)$, a \emph{chain of bricks} of $A$ is an injective map $\Psi: I \rightarrow \brick A$ such that if $i < j$ in $I$, then $\Hom_A(\Psi(i),\Psi(j))=0$. 
We identify two chains $\Psi_1:I_1\to\brick A$ and $\Psi_2:I_2\to\brick A$ if there exists an isomorphism $\rho:I_1\to I_2$ of posets satisfying $\Psi_1=\Psi_2\circ\rho$.
Let $\chbrick A$ denote the set of all chains of bricks of $A$ (after identification). We consider a partial order on $\chbrick A$, as follows: we write $\Psi_1\le\Psi_2$ if there exists an order preserving map $\rho:I_1\to I_2$ satisfying $\Psi_1=\Psi_2\circ\rho$.

To every chain of bricks of $A$, one can associate a unique chain of torsion classes. Before we describe this correspondence, recall that a subset $J$ of an ordered set $I$ is called an \emph{ideal} if $J$ is down-closed, that is, for each $l\in I$, if $l \leq j$, for some $j\in J$, then $l \in J$. By $\mathfrak{J}(I)$ we denote the set of all ideals of $I$ and view it as a poset with respect to the inclusion of ideals.

Let $\Psi: I \rightarrow \brick A$ be a chain of bricks of $A$. For $J \in \mathfrak{J}(I)$, let $T_J\in\tors A$ be the smallest torsion class in $\modu A$ that contains each brick $\Psi(j)$, for $j\in J$. Moreover, if $J \subsetneq J'$ are two ideals in $I$, for any $j \in J'\setminus J$ the brick $\Psi(j)$ belongs to $\mathcal{T}_{J'} \cap \mathcal{T}_{J}^{\perp}$.
As explained by Demonet, in \cite[Appendix]{KD}, this correspondence gives a strictly increasing map of posets $\mathfrak{J}(I)\rightarrow \tors A$, which sends $J$ to $\mathcal{T}_J$.
Thus, any (equivalence class of) chain of bricks of $A$, say $\Psi$, induces a subposet of $\tors A$, which we denote by $\tors_{\Psi} A$.

\medskip

The following result gives a beautiful and direct link between the chains of bricks and chains of torsion classes. Recall that $\chbrick A$ is a poset consisting of all (equivalence classes of) chains of bricks of $A$, with respect to the order described above.
Similarly, $\chtors A$ denotes the poset of all chains of torsion classes in $\tors A$, with the natural order given via refinement of chains.

\begin{theorem}[Demonet {\cite[Appendix]{KD}}]\label{Thm: Demonet's correspondence}
With the above notation, the map
$$f: \chbrick A \rightarrow \chtors A$$
that sends $\Psi$ to $\tors_{\Psi} A$ is an injective map of posets. In particular, $f$ sends the maximal elements of $\chbrick A$ to the maximal elements of $\chtors A$. 
\end{theorem}

We note that the map $f$ in the above theorem is almost never invertible. In fact, $f$ is bijective if and only if $A$ is local. We remark that starting from a chain of bricks $\{B_i\}_{i\in I}$ in $\chbrick A$, we often have that the cardinality of the corresponding chain of torsion classes in $\chtors A$ induced by $f$ is no smaller (and often larger) than the cardinality of $I$. Moreover, one immediately concludes that the length of each chain in $\tors A$ is bounded by $|{\brick A}|$. Also, note that $A$ is brick-infinite if and only if $\tors A$ contains an infinite chain. For more details on the chain of torsion classes and their interactions with bricks, see \cite{BCZ}, \cite{DIJ}, \cite{DI+}, and {\cite[Appendix]{KD}}.

\resumetoc

\section{Brick-splitting torsion pairs}\label{Subsection: Brick-splitting pairs}

As a modern brick-analogue of the classical splitting phenomenon, we say that a torsion pair $(\mathcal{T},\mathcal{F})$ in $\modu A$ is \emph{brick-splitting} if, for each $M \in \brick A$, we have either $M \in \mathcal{T}$ or $M \in \mathcal{F}$. That being the case, $\mathcal{T}$ is called a brick-splitting torsion class and $\mathcal{F}$ is called a brick-splitting torsion-free class. 
If $(\mathcal{T},\mathcal{F})$ is splitting torsion pair in $\modu A$, then evidently each brick in $\modu A$ belongs to either $\mathcal{T}$ or $\mathcal{F}$. However, as shown in the following, the converse is not true in general. 

\subsection{Characterizations}
In this part, we give several characterizations of brick-splitting torsion pairs. 
We begin with the following proposition, which characterizes brick-splitting torsion pairs in terms of the intervals that each $\mathcal{T}\in \tors A$ specifies. This particularly implies equivalence of the first two assertions of Theorem \ref{Thm: brick-splitting torsion class characterization}.
For each $X$ in $\modu A$, recall that $\mathcal{T}(X)$ denotes the smallest torsion class $\mathcal{T} \in \tors A$ with $X \in \mathcal{T}$. Dually $\mathcal{F}(X)$ is defined.

\begin{proposition}\label{Prop: brick-splitting and intervals}
For an algebra $A$, and each torsion class $\mathcal{T}$ in $\modu A$, the following are equivalent:
\begin{enumerate}
    \item $(\mathcal{T}, \mathcal{T}^\perp)$ is a brick-splitting torsion pair;
    \item Each $X\in \brick A$ appears as an arrow in $\Hasse [0,\mathcal{T}]$ or $\Hasse[\mathcal{T}, \modu A]$.
\end{enumerate}
\end{proposition}
\begin{proof}
Assume that (1) holds. Consider $X \in \brick A$. Since $(\mathcal{T}, \mathcal{T}^\perp)$ is brick-splitting, we get $X \in \mathcal{T}$ or $X \in \mathcal{T}^\perp$. In the first case, $\mathcal{T}(X)$ is a torsion class in the interval $[0, \mathcal{T}]$, and we have $X$ is the labeling brick of the minimal inclusion $\mathcal{T}(X) \cap{}^\perp X \subset \mathcal{T}(X)$. In the second case, if $X \in \mathcal{T}^\perp$, then $X$ is the labeling brick of the minimal inclusion $\mathcal{F}(X) \cap X^\perp \subset \mathcal{F}(X)$ of torsion-free classes, all included in $\mathcal{T}^\perp$. This translates to $X$ being a brick-labeling of the corresponding inclusion $^\perp X \subset {^\perp(\mathcal{F}(X) \cap X^\perp})$ of torsion classes in  $[\mathcal{T},\modu A]$. 

Now, assume that $(2)$ holds. Let $X \in \brick A$. If $X$ is a labeling brick between two torsion classes in $[0, \mathcal{T}]$, then obviously $X \in \mathcal{T}$. Suppose that $X$ is the labeling brick between some two adjacent torsion classes $\mathcal{T}_1 \subset \mathcal{T}_2$ in the interval $[\mathcal{T}, \modu A]$. This implies that $X \in \mathcal{T}_1^\perp \subseteq \mathcal{T}^\perp$.
\end{proof}

\begin{remark}
For brick-finite algebras, the above proposition gives an easy criterion to verify whether a given torsion class admits a brick-splitting torsion pair. 
For instance, in Example \ref{Example: Non-trivial brick-splitting torsion pairs}, one can check that every $X\in\brick A$ appears on the left chain of the lattice $\tors A$ (see Figure \ref{Fig. for Example}). Consequently, every torsion class on this chain gives rise to a brick-splitting torsion pair. In contrast, the labeling bricks on the other maximal chain on the right side consists of only two bricks, whereas $\brick A$ has three elements. Thus, we can immediately conclude that $\mathcal{T}''=\Fac(\begin{smallmatrix} 2\\2\end{smallmatrix})$ is not a brick-splitting torsion class. In particular, ${\begin{smallmatrix} 1\\2\end{smallmatrix}} \in \brick A$ does not appear in $[0,\mathcal{T}'']$ nor $[\mathcal{T}'',\modu A]$.
\end{remark}

From the previous proposition and remark, we obtain the following corollary. As we shall see later, this result can be strengthened to give a characterization of a special family of algebras (see Theorem \ref{Thm: Brick-directed & maximal chain of brick-splitting pairs}).

\begin{corollary}
Let $A$ be a brick-finite algebra. If $\Hasse(\tors A)$ has a path of length $|{\brick A}|$, then every torsion class on that path is brick-splitting.
\end{corollary}
\begin{proof}
Labeling bricks along a path in the Hasse quiver are all non-isomorphic. The condition therefore implies that all bricks appear as labels of this path. The statement follows from Proposition \ref{Prop: brick-splitting and intervals}.
\end{proof}

In the next proposition, we give a purely lattice-theoretical characterization of brick-splitting pairs. In particular, together with Proposition \ref{Prop: brick-splitting and intervals}, the following result completes the proof of Theorem \ref{Thm: brick-splitting torsion class characterization}.

\begin{proposition}\label{Prop: brick-splitting is left modular}
For an algebra $A$, and each torsion class $\mathcal{T}$ in $\modu A$, the following are equivalent:
\begin{enumerate}
    \item $(\mathcal{T}, \mathcal{T}^\perp)$ is a brick-splitting torsion pair;
    \item $\mathcal{T}$ is a left modular element of the lattice $\tors A$.
\end{enumerate}
\end{proposition}

\begin{proof}
Assume that $\mathcal{T}$ is left modular. Let $L \in \brick A$ be not in $\mathcal{T}$. We need to prove that $L \in \mathcal{T}^\perp$. Since $L \not \in \mathcal{T}$, the smallest torsion class $\mathcal{T}(L)$ containing $L$ is not contained in $\mathcal{T}$. Since $\mathcal{T}(L)$ is completely join-irreducible, there is a torsion class $\mathcal{Y}$ with $\mathcal{Y} \subsetneq \mathcal{T}(L)$ such that every torsion class that is properly contained in $\mathcal{T}(L)$ is contained in $\mathcal{Y}$. Using that $\mathcal{T}$ is a left modular element, we get
$$(\mathcal{Y} \vee \mathcal{T}) \cap \mathcal{T}(L) = \mathcal{Y} \vee (\mathcal{T} \cap \mathcal{T}(L)).$$
Since $\mathcal{T} \cap \mathcal{T}(L)$ is properly contained in $\mathcal{T}(L)$, we see that $ \mathcal{Y} \vee (\mathcal{T} \cap \mathcal{T}(L)) = \mathcal{Y}$. Thus $(\mathcal{Y} \vee \mathcal{T}) \cap \mathcal{T}(L)=\mathcal{Y} \neq \mathcal{T}(L)$ holds. This means that $\mathcal{T}(L)$ is not contained in $\mathcal{Y} \vee \mathcal{T}$, so that $L \not \in \mathcal{Y} \vee \mathcal{T}.$ Now, we consider the canonical sequence
$$0 \to T \to L \to F \to 0$$
for $L$ where $T \in \mathcal{T}$ and $F \in \mathcal{T}^\perp$. Since $L$ is a minimal extending brick for $\mathcal{Y}$, we know that every proper quotient of $L$ has to be in $\mathcal{Y}$. So if $T$ is non-zero, then $F \in \mathcal{Y}$, making $L \in \mathcal{Y} \vee \mathcal{T}$, a contradiction. Therefore, $T=0$ and $L \in \mathcal{T}^\perp$.

For the other direction, assume that $\mathcal{T}$ is brick-splitting. Observe that it follows from \cite[Theorem 1.3]{DI+} that $\tors A$ is weakly atomic. By the proof of \cite[Lemma 2.8]{TW} and weakly atomic property of $\tors A$, in order to prove that $\mathcal{T}$ is left modular, it suffices to prove that for two torsion classes $\mathcal{V} \subsetneq \mathcal{U}$ forming a cover relation, we have 
$$(*):\quad (\mathcal{U}\cap \mathcal{T}) \vee \mathcal{V} = \mathcal{U} \cap (\mathcal{T} \vee \mathcal{V}).$$
Observe that the containment from left to right is clear. Now, let $L$ be the labeling brick corresponding to the cover relation $\mathcal{V} \subsetneq \mathcal{U}$. Then we have $\mathcal{T}(L)\vee\mathcal{V}=\mathcal{U}$ and ${}^\perp L\cap\mathcal{U}=\mathcal{V}$. Since $\mathcal{T}$ is brick-splitting, either $L \in \mathcal{T}$ or $L \in \mathcal{T}^\perp$. If $L \in \mathcal{T}$, we see that
$$(\mathcal{U} \cap \mathcal{T})\vee\mathcal{V} \supseteq \mathcal{T}(L) \vee \mathcal{V}= \mathcal{U} \supseteq \mathcal{U} \cap (\mathcal{T} \vee \mathcal{V}).$$
Assume now that $L \in \mathcal{T}^\perp$. 
Since both $\mathcal{T}$ and $\mathcal{V}$ are contained in $^\perp L$, we have
$$\mathcal{U} \cap (\mathcal{T} \vee \mathcal{V}) \subseteq \mathcal{U} \cap {^\perp L} = \mathcal{V}\subseteq(\mathcal{U}\cap \mathcal{T}) \vee \mathcal{V}.$$
Therefore, $(*)$ also holds in this case.
\end{proof}

\subsection{Brick quivers and brick-splitting torsion classes}\label{Subsection: Brick-quiver}

To any algebra, one can explicitly associate a quiver that captures the interactions between the bricks and hence the brick-directedness of the algebra under consideration. In particular, for a connected algebra $A$, by $Q^b(A)$ we denote the \emph{brick-quiver} of $A$, constructed as follows: The vertices of $Q^b(A)$ are in bijection with elements of $\brick A$, and for each pair of distinct (hence non-isomorphic) modules $X$ and $Y$ in $\brick A$, we put an arrow from $X$ to $Y$ if $\Hom_A(X,Y)\neq 0$. For an explicit example, see Example \ref{Example: windwheel alg of rank 2}. One motivation for studying the brick-quiver is the following.

\begin{proposition}
There exists a bijection between the following sets.
\begin{enumerate}
\item Brick-splitting torsion classes of $A$;
\item Successor-closed subsets of $Q^b(A)$.
\end{enumerate}
\end{proposition}

\begin{proof}
Given a brick-splitting torsion class $\mathcal{T}$, we associate to it the set $\mathcal{T} \cap \brick A$, which is regarded as a set of vertices of $Q^b(A)$. If $B \to B'$ is an arrow in $Q^b(A)$ with $B \in \mathcal{T}$, then $\Hom_A(B, B') \ne 0$ so that $B' \not \in \mathcal{T}^\perp$. Since $\mathcal{T}$ is brick-splitting, this yields that $B' \in \mathcal{T}$. Hence, the set $\brick(A) \cap \mathcal{T}$ is successor-closed in $Q^b(A)$. Since torsion classes are determined by the bricks they contain, the assignment $\mathcal{T} \mapsto \mathcal{T} \cap \brick A$ is injective. 

Conversely, let $\mathcal{S}$ be a set of bricks that forms a successor-closed set in $Q^b(A)$. We let $\mathcal{T} = \mathcal{T}(\mathcal{S})$ be the smallest torsion class that contains $S$. We observe that for a brick $B$, we have $\Hom_A(\mathcal{S},B)=0$ if and only if $B \in \mathcal{T}^\perp$, if and only if $B \not \in \mathcal{S}$, where the latter equivalence follows from the fact that $\mathcal{S}$ is successor-closed in $Q^b(A)$. It follows that $\mathcal{T} \cap \brick A = \mathcal{S}$. Finally, we check that the constructed $\mathcal{T}$ is brick-splitting. Let $B \in \brick A$. If $B \in \mathcal{S}$, then $B \in \mathcal{T}$. Otherwise, it follows from what we have shown that $B \in \mathcal{T}^\perp$. Hence, $\mathcal{T}$ is brick-splitting. 
\end{proof}

The following proposition lists some properties of the brick quiver that follows from the above construction.

\begin{prop}
Let $A$ be a connected algebra. Then, the brick quiver $Q^b(A)$ satisfies the following properties:
    \begin{enumerate}
    \item $Q^b(A)$ is connected and has no loops.
    \item $Q^b(A)$ has finitely many arrows if and only if it has finitely many vertices. This is the case if and only if $A$ is brick-finite.
    \item $Q^b(A)$ is infinite if and only if it has at least two vertices of infinite degree.
    \item $Q^b(A)$ has a cycle if and only if there is a brick-cycle in $\modu A$.
    \item Every acyclic path in $Q^b(A)$ gives a chain of bricks of $A$, provided no subpath of it can be completed to a cycle.
\end{enumerate}
\end{prop}

\begin{proof} 
Recall that $S_1, \ldots, S_n$ denote a complete list of (isomorphism classes of) simple modules in $\modu A$. In particular, each $S_i$ appears as a vertex of $Q^b(A)$. 
\begin{enumerate}
    \item First, observe that the top and socle of each $A$-module are given as direct sums of simple modules. In particular, for $X\in \brick A$, and each simple module $S$ appearing as a summand of $\topm(X)$, in the brick quiver $Q^b(A)$ there is an arrow $X \to S$ (dually, for every simple summand $S'$ of $\soc(X)$, there is an arrow $S' \to X$). Thus, each non-simple $X$ in $\brick A$ is connected to at least one simple module in $Q^b(A)$. 
    Next, we note that for each arrow $\alpha \in \Omega_1$ which is not a loop, say $\alpha: i \to j$, we have a brick $M$ of length $2$ given by the middle term of a non-split short exact sequence
    $$0 \to S_j \to M \to S_i \to 0.$$ Then, in $Q^b(A)$, we obviously have the arrows $S_j \to M$, and $M \to S_i$. Hence, using bricks of length $2$, we conclude that $Q^b(A)$ is connected. 
    The fact that $Q^b(A)$ has no loops immediately follows from the definition.

    \item As explained in part (1), for each non-simple $X$ in $\brick A$, in $Q^b(A)$ there is at least one arrow from $X$ to one of the simple modules $S_1, \ldots, S_n$. Thus, $Q^b(A)$ has infinitely many arrows if and only if $\brick A$ contains infinitely many (non-simple) bricks. 

    \item If $Q^b(A)$ is an infinite quiver, then obviously $\brick A$ contains infinitely many (non-simple) bricks, because $Q^b(A)$ has finitely many arrows starting at $i$ and ending at $j$ for any pair $(i,j)$ of vertices.  Consequently, for some $1\leq i \leq n$, the simple module $S_i$ appears as a summand of the socle of infinitely many bricks in $\brick A$, say $\{X_t\}_{t\in \mathbb{Z}_{>0}}$. Hence, in $Q^b(A)$, there are infinitely many arrows outgoing from $S_i$, namely, $S_i \to X_t$, for each ${t\in \mathbb{Z}_{>0}}$. Since every $X_t$ is a brick, then $S_i$ cannot be a summand of $\topm(X_t)$, for any ${t\in \mathbb{Z}_{>0}}$. Hence, there exists $1\leq j \leq n$, with $j\neq i$, such that the simple module $S_j$ appears as a summand of the top of infinitely many bricks in $\{X_t\}_{t\in \mathbb{Z}_{>0}}$. Thus, in $Q^b(A)$, there are infinitely many arrows incoming to $S_j$. From this, we conclude that $S_i$ and $S_j$ in  $Q^b(A)$ are of infinite degree. The converse is trivial.

    \item This directly follows from the definition.

    \item We only prove the statements for finite paths in $Q^b(A)$, because a similar argument also works for the infinite paths. 
    Let $X_0 \to X_1 \to \cdots \to X_m$ be an acyclic path in $Q^b(A)$, and assume that no subpath of it can be completed to a cycle. Namely, for each $0 \leq i < j \leq m$, we have $X_i \not\simeq X_j$, and furthermore, there is no path in $Q^b(A)$ from $X_j$ to $X_i$. In particular, $\Hom_A(X_j,X_i)=0$. To show that $X_0 \to \cdots \to X_m$  gives rise to a chain of bricks, we consider the totally ordered set $(I,\leq)$, where $I:=\{0,1, \cdots,m\}$ is the index set of the bricks $X_i$ in the acyclic path $X_0 \to \cdots \to X_m$, and the order in $(I,\leq)$ is defined to be $s < t$ if $\Hom_A(X_t,X_s)=0$. 
    Then, it is immediate that the injective map $\Psi: I \rightarrow \brick A$, given by $\Psi(i)=X_{m-i}$, specifies a chain of bricks, particularly because, for all $s < t$ in $I$, we have $\Hom_A(\Psi(s),\Psi(t))=0$.
\end{enumerate}
\end{proof}

In the following remark, we highlight a property of the brick-quiver of algebras. 

\begin{remark}
As observed in \cite{MP2}, from some earlier results of \cite{As}, it follows that for any brick-finite algebra $A$ of rank $n$, there is no set of pairwise Hom-orthogonal modules of size $n+1$. In particular, for such an algebra, the maximal size of an independent set in $Q^b(A)$ is $n$, that is, the independence number of $Q^b(A)$ is $n$ and is attained exactly for the set of vertices in $Q^b(A)$ corresponding to the simple $A$-modules (see \cite[Theorem A]{MP2}). In light of the Semibrick conjecture (see \cite[Conjecture 2.2]{MP2}), it is expected that the converse of the latter statement is true, that is, the independence number of $Q^b(A)$ is $n$ if and only if $A$ is brick-finite. This is already known to be the case for various families of algebras, but it remains open in full generality (for more details, see \cite{MP2, MP4} and references therein).
\end{remark}

\subsection{Basic properties}
We begin with the simple observation that if, for an algebra $A$, each indecomposable $A$-module is a brick, then the notions of splitting and brick-splitting torsion pairs coincide. Those algebras that satisfy the equality $\ind A=\brick A$ are shown to be exactly the family of locally representation-directed algebras, which are known to properly generalize the family of representation-directed algebras (i.e., those algebras for which there is no cycle in $\modu A$). For different characterizations of locally representation-directed algebras, see \cite{Dr} and \cite{MP1}. In particular, if $A$ is locally representation-directed, then every brick-splitting torsion pair in $\modu A$ is in fact splitting. 
We will frequently use the following lemma in some arguments. It is probably well known to experts, but we include a proof for completeness.

\begin{lemma} \label{Lemma: brick-splitting implies no Ext between bricks}
    Let $(\mathcal{T}, \mathcal{F})$ be brick-splitting. 
    If $X\in\mathcal{T}$ and $Y\in\mathcal{F}$ are bricks with $\Hom_A(Y,X)=0$, then $\Ext^1_A(Y,X)=0$.
\end{lemma}

\begin{proof}
Since $\{X,Y\}$ is a semibrick (i.e., a set of pairwise Hom-orthogonal bricks), we obtain a wide subcategory $\mathcal{A}=\Filt(X\oplus Y)$.
Assume that there is a non-split short exact sequence
    $$0 \to X \stackrel{u}{\to} E \stackrel{v}{\to} Y \to 0.$$
Then $E$ is an indecomposable object of length $2$ in the abelian category $\mathcal{A}$ whose composition factors are non-isomorphic. Thus any non-zero endomorphism of $E$ has to be an isomorphism, and hence $E$ is a brick. 
By the brick-splitting property, either $E \in \mathcal{T}$ or $E \in \mathcal{F}$. In the first case, we get $v=0$, a contradiction. In the second case, we get $u=0$, a contradiction.
\end{proof}

In general, a brick-splitting torsion pair is not necessarily splitting (see Example \ref{Example: Non-trivial brick-splitting torsion pairs}). For hereditary algebras, the following proposition shows that when the torsion pair is functorially finite, these two notions of splitting coincide. Before stating the proposition, we need the following lemma, which can be found in \cite[Chapter IV, Lemma 1.5]{Ha}.

\begin{lemma} \label{Lemma: Happel}
    Let $A$ be a hereditary $k$-algebra. Let $X,Y$ be indecomposable rigid $A$-modules such that $\Ext^1_A(X,Y)=0$. Then at least one of $\Hom_A(Y,X)$ and $\Ext^1_A(Y,X)$ is zero.
\end{lemma}
 
\begin{prop}
Let $A$ be a hereditary algebra and $(\mathcal{T},\mathcal{F})$ be a functorially finite torsion pair in $\modu A$. Then, $(\mathcal{T},\mathcal{F})$ is splitting if and only if it is brick-splitting.
\end{prop}

\begin{proof}
We only need to prove one direction. Assume  $(\mathcal{T},\mathcal{F})$ is brick-splitting. 
Since $(\mathcal{T},\mathcal{F})$ is functorially finite, by \cite[Proposition 2.9]{As}, we take the left finite semibrick $\mathcal{B}$ such that $\mathcal{T}=\mathcal{T}(\mathcal{B})$, and the right finite semibrick $\mathcal{C}$ such that $\mathcal{F}=\mathcal{F}(\mathcal{C})$.

We show that each $X\in\mathcal{B}\cup\mathcal{C}$ is rigid. Let $\mathcal{S}:=\mathcal{B}$ if $X\in\mathcal{B}$, and $\mathcal{S}:=\mathcal{C}$ if $X\in\mathcal{C}$.
By \cite[Theorem 2.27]{As} and its dual, we have that $\Filt \mathcal{S}$ is equivalent to the module category of some finite dimensional algebra $\Gamma$. We can take an idempotent $e \in \Gamma$ such that $\Filt X \simeq \modu \Gamma/\langle e \rangle$, since $\Filt X$ is a Serre subcategory of $\Filt\mathcal{S}$. Denote by $Y$ the indecomposable projective object in $\Filt X$. Then, $\Ext^1_A(Y,X)=0$. 
Since  $Y\in\Filt X$, we have a monomorphism $\iota: X \to Y$. Because $A$ is hereditary, this gives rise to an epimorphism $0=\Ext^1_A(Y,X) \to \Ext^1_A(X,X)$, showing that $X$ is rigid.

Let us fix $B \in \mathcal{B}$ and $C \in \mathcal{C}$. We have $\Ext_A^1(B,C)=0$, because $\mathcal{B} \cup \mathcal{C}[1]$ is a 2-term simple-minded collection by \cite[Theorem 3.3]{As}. 
Then by Lemma \ref{Lemma: Happel}, we get $\Hom_A(C,B)=0$ or $\Ext_A^1(C,B)=0$. In the former case, since $B \in \mathcal{T}$ and $C \in \mathcal{F}$, we have $\Ext_A^1(C,B)=0$ by Lemma \ref{Lemma: brick-splitting implies no Ext between bricks}. Hence, we have $\Ext^1_A(C,B) = 0$ in both cases. This shows that $\Ext^1_A(\mathcal{C}, \mathcal{B})=0$. 
Then, the hereditary property yields that $\Ext^1_A(\mathcal{F}, \mathcal{T})=\Ext^1_A(\mathcal{F}(\mathcal{C}), \mathcal{T}(\mathcal{B}))=0$. 
Now, the desired result follows from Proposition \ref{Prop: splitting pairs}.
\end{proof}

\begin{prop}
    If $A$ is hereditary of tame type then a torsion pair in $\modu A$ is splitting if and only if it is brick-splitting.
\end{prop}

\begin{proof}
If $A$ is a representation-finite hereditary algebra, it is well known that $\ind A=\brick A$. Hence, there is nothing to show. 
Thus, henceforth we assume $A$ is representation-infinite and let $(\mathcal{T}, \mathcal{F})$ be a brick-splitting torsion pair. 
Suppose $M$ is an indecomposable which is not a brick. Then $M$ lies in some stable tube in $\Gamma_A$. The quasi-simple modules of the tubes are all bricks. If one of those quasi-simples lies in $\mathcal{F}$ and another lies in $\mathcal{T}$, then there is a short exact sequence
    $$0 \to L \to E \to N \to 0$$
    with $L,N$ quasi-simples where $L \in \mathcal{T}$ and $N \in \mathcal{F}$. Lemma \ref{Lemma: brick-splitting implies no Ext between bricks}, together with the assumption that $L$ and $N$ are Hom-orthogonal, leads to the desired contradiction.  
    Hence, for that tube, either all quasi-simples are in $\mathcal{T}$ or they all belong to $\mathcal{F}$. This forces $M$ to be in $\mathcal{T}$ or in $\mathcal{F}$. Therefore, the pair $(\mathcal{T}, \mathcal{F})$ is splitting.
\end{proof}

The preceding paragraphs show that to find an algebra which admits a brick-splitting torsion pair which is not splitting, one must search outside the two families of tame hereditary algebras and locally representation-directed algebras. Such an example is given below. 
In fact, the following example shows that an algebra may have no non-trivial splitting torsion pairs but it can admit non-trivial brick-splitting torsion pairs. In particular, the notion of brick-splitting pairs is novel. Before we present the following example, let us recall that a module $X$ in $\modu A$ is said to be \emph{$\tau$-rigid} if $\Hom_A(X,\tau_A X)=0$. By $\textit{i}\taurigid A$ we denote the set of all (isomorphism classes) of indecomposable $\tau$-rigid modules in $\modu A$.

\begin{example}\label{Example: Non-trivial brick-splitting torsion pairs}
Let $Q$ be the following quiver
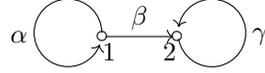
\begin{center}
    \begin{tikzpicture}
    \draw[->] (2.5,2.55) arc (0:340:0.45cm);
\node at (1.4,2.5) {$\alpha$};
        \node at (3.5,2.5) {$\circ$};
        \node at (3.4,2.3) {$2$};
   \draw [->] (2.55,2.5) --(3.45,2.5) ; 
\node at (3,2.75) {$\beta$};
    \node at (2.5, 2.5) {$\circ$}; 
    \node at (2.6, 2.3) {$1$}; 
    \draw[->] (3.52,2.47) arc (190:530:0.45cm);
\node at (4.6,2.5) {$\gamma$};
    \end{tikzpicture}
\end{center}
and define $A:=kQ/I$, where $I$ is generated by all paths of length $2$.
In particular, $A$ is a radical square zero string algebra. Thus, we can describe the indecomposable $A$-modules in terms of the strings in $(Q,I)$, and we have
$$\ind A=\{e_1, e_2, \alpha, \beta, \gamma, \beta \alpha^{-1}, \gamma^{-1}\beta, \gamma^{-1}\beta\alpha^{-1} \}.$$
If the indecomposable $A$-modules corresponding to the above strings are presented in terms of their radical filtration, we obtain
$$\ind A=\{{\begin{smallmatrix} 1\end{smallmatrix}}, {\begin{smallmatrix} 2\end{smallmatrix}}, {\begin{smallmatrix} 1\\1\end{smallmatrix}}, {\begin{smallmatrix} 1\\2\end{smallmatrix}}, {\begin{smallmatrix} 2\\2\end{smallmatrix}}, {\begin{smallmatrix} & 1& \\2& &1\end{smallmatrix}}, {\begin{smallmatrix} 2& &1 \\ &2& \end{smallmatrix}}, {\begin{smallmatrix} 2& &1& \\ &2& &1 \end{smallmatrix}} \}.
$$
Here, the simple modules $S_1$ and $S_2$ correspond to the strings of zero length  $e_1$ and $e_2$, respectively associated to vertices $1$ and $2$, and the indecomposable projective modules $P_1$ and $P_2$ respectively correspond to the strings $\beta\alpha^{-1}$ and $\gamma$.

One can easily check that bricks in $\modu A$ correspond to the strings $e_1, e_2$, and  $\beta$. That is, $\brick A=\{{\begin{smallmatrix} 1\end{smallmatrix}}, {\begin{smallmatrix} 2\end{smallmatrix}}, {\begin{smallmatrix} 1\\2\end{smallmatrix}} \}$. Moreover, ${\begin{smallmatrix} 2\\2\end{smallmatrix}}$ and ${\begin{smallmatrix} & 1& \\2& &1\end{smallmatrix}}$ respectively correspond to $P_2$ and $P_1$, which are evidently $\tau$-rigid. Also, a simple computation shows that $\tau({\begin{smallmatrix} 1\\1\end{smallmatrix}})= {\begin{smallmatrix} 2\\2\end{smallmatrix}}$, hence, $\Hom_A({\begin{smallmatrix} 1\\1\end{smallmatrix}}, \tau({\begin{smallmatrix} 1\\1\end{smallmatrix}}))=0$. In fact, the indecomposable $\tau$-rigid modules are specified by the strings $\alpha, \gamma$ and $\beta \alpha^{-1}$. Thus,  $\textit{i}\taurigid A=\{ {\begin{smallmatrix} 1\\1\end{smallmatrix}},{\begin{smallmatrix} 2\\2\end{smallmatrix}}, {\begin{smallmatrix} & 1& \\2& &1\end{smallmatrix}} \}$.

Since $A$ is representation-finite, there are only finitely many torsion classes. As depicted in Figure 1, the lattice structure of $\tors A$ can be described in terms of a finite $2$-regular directed graph. In particular, each vertex in Figure 1 denotes a torsion class $\mathcal{T}$ in $\tors A$, labeled by the $\tau$-rigid $A$-module $M$ that generates $\mathcal{T}$, that is, $\mathcal{T}:=\Fac M$. Moreover, each arrow in Figure 1 indicates a cover relation in $\tors A$ induced by the inclusion.

Let $\mathcal{T}$ be the torsion class labeled by $\{{\begin{smallmatrix} & 1& \\2& &1\end{smallmatrix}}\}$, that is, $\mathcal{T}:=\Fac({\begin{smallmatrix} & 1& \\2& &1\end{smallmatrix}})$. Then, the set of indecomposables in $\mathcal{T}$ is $\{{\begin{smallmatrix} 1\end{smallmatrix}}, {\begin{smallmatrix} 1\\1\end{smallmatrix}}, {\begin{smallmatrix} 1\\2\end{smallmatrix}}, {\begin{smallmatrix} & 1& \\2& &1\end{smallmatrix}}\}$. Put $\mathcal{F}:=\mathcal{T}^\perp$, and observe that the set of indecomposables in $\mathcal{F}$ is $\{{\begin{smallmatrix} 2\end{smallmatrix}}, {\begin{smallmatrix} 2\\2\end{smallmatrix}}\}$. 
Evidently, $(\mathcal{T},\mathcal{F})$ is a brick-splitting torsion pair in $\modu A$. However, $(\mathcal{T},\mathcal{F})$ is not splitting, particularly because ${\begin{smallmatrix} 2& &1 \\ &2& \end{smallmatrix}}$ and ${\begin{smallmatrix} 2& &1& \\ &2& &1 \end{smallmatrix}}$ do not belong to $\mathcal{T}$ nor $\mathcal{F}$. 

If $\mathcal{T}':=\Fac({\begin{smallmatrix} 1\\1\end{smallmatrix}})$, one can similarly show that the torsion pair $(\mathcal{T}',\mathcal{F}')$ is brick-splitting, but not splitting.

We note that, unlike $\mathcal{T}$ and $\mathcal{T}'$, a simple computation shows that the torsion class labeled by $\{{\begin{smallmatrix} 2\\2\end{smallmatrix}}\}$ is not brick-splitting. More precisely, if $\mathcal{T}'':=\Fac({\begin{smallmatrix} 2\\2\end{smallmatrix}})$, and $\mathcal{F}'':=\mathcal{T}''^\perp$, then the brick ${\begin{smallmatrix} 1\\2\end{smallmatrix}}$ does not belong to $\mathcal{T}''$, nor $\mathcal{F}''$.

Finally, observe that $\modu A$ has no non-trivial splitting torsion pair. 
\end{example} 

\begin{remark}
Primarily motivated by $\tau$-tilting theory, and more specifically via ``brick-$\tau$-rigid correspondence" established in \cite{DIJ}, recent studies of bricks and indecomposable $\tau$-rigid modules have been developed parallel to each other. In that regard, analogous to the notion of brick-splitting torsion pairs, one can naturally define the notion of $\tau$-splitting torsion pairs: $(\mathcal{T},\mathcal{F})$ in $\modu A$ is a \emph{$\tau$-splitting} torsion pair if each indecomposable $\tau$-rigid $A$-module belongs to $\mathcal{T}$ or $\mathcal{F}$.

Although we do not study the notion of $\tau$-splitting pairs in this work, we point out that the notions of brick-splitting and $\tau$-splitting are different. 
In particular, for Example \ref{Example: Non-trivial brick-splitting torsion pairs}, recall that $(\mathcal{T}', \mathcal{F}')$ is a brick-splitting pair, where $\mathcal{T}'=\Fac({\begin{smallmatrix} 1\\1\end{smallmatrix}})$. However, once can easily show that this pair does not split the set $\textit{i}\taurigid A=\{ {\begin{smallmatrix} 1\\1\end{smallmatrix}},{\begin{smallmatrix} 2\\2\end{smallmatrix}}, {\begin{smallmatrix} & 1& \\2& &1\end{smallmatrix}} \}$, because ${\begin{smallmatrix} & 1& \\2& &1\end{smallmatrix}} \notin \mathcal{T}'\cup \mathcal{F}'$. Namely, $(\mathcal{T}', \mathcal{F}')$ is not a $\tau$-splitting pair.
\end{remark}

\begin{figure}\label{Fig. for Example}
\begin{center}
    \begin{tikzpicture}
\node at (0,4) {$\modu A$};
    \draw [->] (-0.1,3.75)--(-1,3.25); 
\node at (-1,3) {$\{{\begin{smallmatrix} & 1& \\2& &1\end{smallmatrix}}\}$};
    \draw [->] (-1,2.75)--(-1,2.25); 
\node at (-1,2) {$\{{\begin{smallmatrix} 1\\1\end{smallmatrix}}\}$};
    \draw [->] (-1,1.75)--(-0.1,1.25); 
\node at (0,1) {$0$};
    \draw [->] (1,2.25)--(0.1,1.25); 
\node at (1,2.5) {$\{{\begin{smallmatrix} 2\\2\end{smallmatrix}}\}$};
    \draw [->] (0,3.75)--(1,2.75); 
    \end{tikzpicture}
\end{center}
\caption{Lattice of torsion classes in Example \ref{Example: Non-trivial brick-splitting torsion pairs}. A non-trivial torsion class is specified by the $\tau$-rigid module that generates it.}
\end{figure}
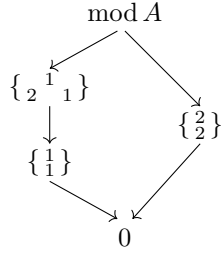

The following proposition shows how brick-splitting torsion pairs can provide us with further information on the structure of the algebras under consideration. Let us recall that if $\mathcal{C}$ is a subcategory of $\modu A$, then $\simp A\cap \mathcal{C}$ denotes the set of non-isomorphic simple $A$-modules belonging to $\mathcal{C}$.

\begin{proposition} \label{Prop: extension between simples}
Let $A$ be a connected algebra and suppose $(\mathcal{T},\mathcal{F})$ is a non-trivial brick-splitting torsion pair in $\modu A$. Then, for each $S_1 \in \simp A\cap \mathcal{T}$ and $S_2 \in \simp A \cap \mathcal{F}$, we have $\Ext^1_A(S_2, S_1)=0$. Hence, in $\Omega$, there is no arrow from the vertices corresponding to $\simp A \cap \mathcal{F}$ to the vertices corresponding to $\simp A \cap \mathcal{T}$.
\end{proposition}

\begin{proof}
Let $S_1 \in \simp A\cap \mathcal{T}$ and $S_2 \in \simp A \cap \mathcal{F}$. Observe that any non-split short exact sequence
$$0 \to S_1 \stackrel{u}{\to} E \stackrel{v}{\to} S_2 \to 0$$
is such that $E$ is a brick, since $S_1, S_2$ are non-isomorphic. By the brick-splitting property, we have $E \in \mathcal{T}\cup \mathcal{F}$. If $E \in \mathcal{T}$, then $v=0$, a contradiction. If $E \in \mathcal{F}$, then $u=0$, a contradiction.
\end{proof}

As discussed below, the preceding proposition has some interesting consequences for certain families of algebras. Before stating such results, we show a useful lemma.

\begin{lemma}
Let $A$ be an algebra and $(\mathcal{T},\mathcal{F})$ be a (brick-)splitting torsion pair in $\modu A$. For each quotient algebra of $A$, say $B$, we have that $(\mathcal{T}\cap \modu B,\mathcal{F}\cap \modu B)$ is a (brick-)splitting torsion pair in $\modu B$.
\end{lemma}
\begin{proof}
Let $B = A/I$. The result follows from the well-known equality $\brick B = \brick A \cap \modu B$. More specifically, bricks over $B$ are exactly the bricks over $A$ that are annihilated by $I$.
\end{proof}

Using Proposition \ref{Prop: extension between simples}, we obtain the following interesting result. As previously defined in Section \ref{Sec: Summary of Main Results}, $A$ is called a \emph{fully cyclic} algebra if for any idempotent $e$ of $A$, from $eA(1-e)=0$ we get $e=0$ or $e=1$. 
Now we prove Corollary \ref{Cor: No brick-splitting over self-injective algebras}.

\begin{corollary}\label{Cor: fully cyclic and self-injective}
If $A$ is fully cyclic, the only brick-splitting torsion pairs in $\modu A$ are trivial. In particular, if $A$ is a connected self-injective algebra, then the only brick-splitting torsion pairs in $\modu A$ are $(0,\modu A)$ and $(\modu A, 0)$.
\end{corollary}

\begin{proof}
Assume that there is a non-trivial brick-splitting torsion pair $(\mathcal{T}, \mathcal{F})$ in $\modu A$. Then, each of $\simp A \cap \mathcal{F}$ and $\simp A \cap \mathcal{T}$ is non-empty. 
If $\Omega_0$ denotes the vertex set of $\Omega$, then it follows from Proposition \ref{Prop: extension between simples} that we can partition $\Omega_0$ into two non-empty subsets $\mathcal{S}_1$ and $\mathcal{S}_2$ such that there is no arrow from a vertex of $\mathcal{S}_1$ to a vertex of $\mathcal{S}_2$. If $e$ denotes the idempotent corresponding to the vertices belonging to $\mathcal{S}_1$, we have $(1-e)Ae=0$. Since $e\neq 0$ and $e\neq 1$, then $A$ is not fully cyclic. This is the desired contradiction, and we conclude the first assertion.

Now, we check that any connected self-injective algebra is fully cyclic. If we assume otherwise, then we get $eA(1-e)=0$ for some idempotent $0 \ne e \ne 1$. Hence, the projective module $A(1-e)$, which has to be injective, is supported only on the vertices in $\Omega_0$ corresponding to $1-e$. This yields that $D((1-e)A) \cong A(1-e)$. By the connectedness of $A$, we get $(1-e)A(1-e) = A$, a contradiction.
\end{proof}

The following proposition shows how certain chains of brick-splitting torsion classes control the shape of the quiver of the algebra under consideration.

\begin{proposition} \label{prop: chain of brick-splitting separating simples}
    Assume that there exists a chain $$0 = \mathcal{T}_0 \subset \mathcal{T}_1 \subset \mathcal{T}_2 \subset \cdots \subset \mathcal{T}_n = \modu A$$
  of brick-splitting torsion classes where $\mathcal{T}_i$ contains exactly $i$ non-isomorphic simple modules. Then $A$ is weakly triangular.
\end{proposition}

\begin{proof}
Consider a chain $$0 = \mathcal{T}_0 \subset \mathcal{T}_1 \subset \mathcal{T}_2 \subset \cdots \subset \mathcal{T}_n = \modu A$$
of brick-splitting torsion pairs such that $\mathcal{T}_i$ contains exactly $i$ non-isomorphic simple modules. 
For $1 \le i \le n$, let $U_i$ denote the unique simple module in $\mathcal{T}_i$ that is not in $\mathcal{T}_{i-1}$. Then, Proposition \ref{Prop:brick-splitting and simples} implies that $\Ext^1_A(U_i, U_j)=0$, whenever $j < i$, which yields that $A$ is weakly triangular.
\end{proof}

We finish this subsection with a remark on the preceding proposition.

\begin{remark}
    
Note that the converse of Proposition \ref{prop: chain of brick-splitting separating simples} does not hold. For instance, consider the quiver
$$
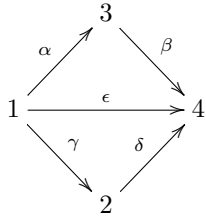
\xymatrix{& 3 \ar[dr]^\beta & \\
1 \ar[rr]^\epsilon \ar[dr]^\gamma \ar[ur]^\alpha & & 4 \\
& 2 \ar[ur]^\delta&}$$
and $A$ the corresponding path algebra. 
As observed in the proof of Proposition \ref{prop: chain of brick-splitting separating simples}, if there is a chain $$0 = \mathcal{T}_0 \subset \mathcal{T}_1 \subset \mathcal{T}_2 \subset \mathcal{T}_3 \subset \mathcal{T}_4 = \modu A$$ of brick-splitting torsion classes, where $\mathcal{T}_i$ contains exactly $i$ simple modules, then this induces an admissible ordering on the simple modules, or equivalently on the vertices of the quiver. By symmetry, we can assume that $\mathcal{T}_i$ contains exactly the simple modules $S_1, \ldots, S_i$. In particular, $\mathcal{T}_2$ contains exactly the simple modules $S_1, S_2$. 
Consider the module $M$ given by the string $\beta^{-1}\epsilon\gamma^{-1}$. Observe that $M$ is a brick whose top is $S_1 \oplus S_3$ and socle is $S_2 \oplus S_4$. If $M$ lies in $\mathcal{T}_2$, then $S_3 \in \mathcal{T}_2$, a contradiction. If $M$ lies in the corresponding torsion-free class $\mathcal{F}_2$, then $S_2 \in \mathcal{F}_2$, again a contradiction. 
\end{remark}

\section{Brick-directed algebras}\label{Section: Brick-directed algebras}
For an algebra $A$, a \emph{path} in $\modu A$ is a sequence $X_0 \xrightarrow{f_1} X_1 \xrightarrow{f_2}  \cdots X_{t-1} \xrightarrow{f_t}  X_{t}$, where $t$ is a positive integer, $X_0, \cdots , X_t$ are indecomposables, and for $1\leq i \leq t$, $f_i:X_{i-1}\rightarrow X_i$ is a non-zero non-invertible morphism in $\modu A$. Such a path is called a \emph{cycle} if $X_0$ and $X_t$ are isomorphic. We note that in a path in $\modu A$, some of the morphisms may belong to $\rad^{\infty}(A)$, that is to say, a path in $\modu A$ is not necessarily a path in $\Gamma_A$. 

Recall that $X \in \ind (A)$ is said to be \emph{directing} if it does not lie on any cycle in $\modu A$. Such modules are historically important. It is known that the preprojective components (similarly, the preinjective components) consist entirely of directing modules.
An algebra $A$ is called \emph{representation-directed} if all indecomposables are directed. Observe that if $A$ is representation-infinite, then $\modu A$ admits a cycle, hence every representation-directed algebra is representation-finite (for instance, see \cite[IX. Corollary 3.4]{ASS}). 

\subsection{Characterizations of brick-directedness}\label{Characterizations of brick-directedness}
As introduced in Section \ref{Sec: Summary of Main Results}, in this work we generalize the above classical notion of representation-directedness. In particular, a cycle $X_0\xrightarrow{f_1} X_1\xrightarrow{f_2}  \cdots \xrightarrow{f_{m-1}}  X_{m-1}\xrightarrow{f_m}  X_m=X_0$ in $\modu A$ is called a \emph{brick-cycle} if $X_i \in \brick A$, for all $0\leq i\leq m$. We say an algebra $A$ is \emph{brick-directed} if there exists no brick-cycle in $\modu A$. 
From the definition, it is immediate that each representation-directed algebra is brick-directed. However, there are many brick-directed algebras that are not representation-directed (see Section \ref{Subsection: Construction of new brick-directed algebras}).

As recalled in Theorem \ref{Thm: Demonet's correspondence}, from \cite{KD} we know that each chain of bricks of $A$ gives rise to a chain of torsion classes in $\modu A$, and this correspondence induces a bijection between the equivalence classes of maximal chains of bricks and the set of maximal chains of torsion classes in $\tors A$. In fact, an interesting characterization of brick-directed algebras can be given in terms of their maximal chain of torsion classes and the notion of brick-splitting torsion pairs. We recall that $Q^b(A)$ denotes the brick-quiver of the algebra $A$, as introduced in Section \ref{Subsection: Brick-quiver}.

\begin{theorem}\label{Thm: Brick-directed & maximal chain of brick-splitting pairs}
For an algebra $A$, the following are equivalent:
\begin{enumerate}
    \item[$(i)$] $A$ is brick-directed;
    \item[$(ii)$] $Q^b(A)$ is acyclic.
\end{enumerate}
Moreover, in this case, there is a bijection between the following objects.
\begin{enumerate}
    \item A total order $J$ on $\brick A = \{B_j\}_{j \in J}$ with $\Hom_A(B_{i}, B_j)=0$, whenever $i < j$ in $J$.
    \item A maximal chain $\{\mathcal{T}_i\}_{i\in I}$ in $\tors A$ such that every $\mathcal{T}_i$ is a brick-splitting torsion class.
\end{enumerate}
\end{theorem}

\begin{proof}
From the construction of the brick quiver $Q^b(A)$, it is clear that $(i)$ and $(ii)$ are equivalent. Hence, we construct a map between objects from $(1)$ to objects from $(2)$ as follows. We consider a totally ordered set $J$ where $\brick A = \{B_j\}_{j \in J}$ with $\Hom_A(B_{i}, B_j)=0$, whenever $i < j$ in $J$. We consider the set $\mathcal{I}$ of all down-closed subsets of $J$. Then $\mathcal{I}$ can be totally ordered by inclusion. For $I \in \mathcal{I}$, let $\mathcal{T}_I$ be the smallest torsion class that contains all $B_i$, for $i \in I$. The map is
$$(*):\quad\{B_j\}_{j \in J} \mapsto \{\mathcal{T}_I\}_{I \in \mathcal{I}}.$$
If $j \in J$ is such that $i < j$ for all $i \in I$, then it is clear that $\Hom_A(\mathcal{T}_I, B_j)=0$, hence $B_j \in \mathcal{T}_I^\perp$. This proves that each $\mathcal{T}_I$ is a brick-splitting torsion class. Assume that the constructed chain $\{\mathcal{T}_I\}_{I \in \mathcal{I}}$ is not maximal and let $\mathcal{T}$ be a torsion class that can be added to properly extend this chain. Let $K \subseteq J$ correspond to the bricks in $\mathcal{T}$. From our assumptions, it follows that $K$ does not form a down closed set. Hence, there exist $B_i, B_j \in \brick A$ such that $B_j \in \mathcal{T}$ while $B_i \not \in \mathcal{T}$, with $i < j$. Now, the torsion class $\mathcal{T}_{\le i}$ has to be comparable to $\mathcal{T}$. Since $B_i \not\in \mathcal{T}$, we have $\mathcal{T} \subseteq\mathcal{T}_{\le i}$, which leads to the contradiction $B_j \in \mathcal{T}_{\le i}$. Hence, $\{\mathcal{T}_I\}_{I \in \mathcal{I}}$ is a maximal chain of brick-splitting torsion classes. 

To prove that the above assignment is bijective, we construct an inverse. Let $\{\mathcal{T}_i\}_{i\in I}$ in $\tors A$ be a maximal chain such that every $\mathcal{T}_i$ is a brick-splitting torsion class. Given a brick $B$, it follows from the maximality of the chain that there is a unique $i_B \in I$ with the property that $B \in \mathcal{T}_{i_B}$ and for $j < i_B$, we have $B \not \in \mathcal{T}_j$. Indeed, one takes $i_B = \bigwedge_{B \in \mathcal{T}_i} i$. If we consider the subchain $\{\mathcal{T}_{i_B}\}_{B \in \brick A}$, this induces a total order $\le$ on $\brick A$. If $U,V \in \brick A$ and $U < V$ with respect to this order, then $\mathcal{T}_{i_U} \subset \mathcal{T}_{i_V}$ and it follows that $V \not \in \mathcal{T}_{i_U}$. Since $\mathcal{T}_{i_U}$ is brick splitting, we get $V \in \mathcal{T}_{i_U}^\perp$ so that $\Hom_A(U,V)=0$. Note that with the notation $B_{i_B}=B$, the totally ordered set $\{i_B \mid B \in \brick A\}$ gives us a way to label the bricks. 

We finish the proof by checking that the assignment
$$\{\mathcal{T}_i\}_{i\in I} \mapsto  \{B_j\}_{j \in \{i_B \mid B \in \brick A\}}$$
provides an inverse of $(*)$. 
Given $\{B_j\}_{j \in J}$, the map $(*)$ assigns to it the chain $\{\mathcal{T}_I\}_{I \in \mathcal{I}}$ which is such that $i_B = \{j \in J \mid B_j \le B \}$. Clearly, for $U, V \in \brick A$, we have $U \le V$ if and only if $i_U \subseteq i_V$, which proves that the second map is a left inverse to $(*)$. Conversely, given a maximal chain $\{\mathcal{T}_i\}_{i\in I}$ of brick-splitting torsion classes, the second map assigns to it a total order on the bricks belonging to $\{i_B \mid B \in \brick A\}$. Given $i \in I$, if $B$ is a brick in $\mathcal{T}_i$, then $\mathcal{T}_{i_B} \subseteq \mathcal{T}_i$. Because every torsion class is uniquely determined by the bricks it contains, we get $\mathcal{T}_i = \bigvee_{B \in \mathcal{T}_i} \mathcal{T}_{i_B}$, which means that $i=\bigvee_{i_B \le i} i_B$. This yields that the second map is also a right inverse to $(*)$.
\end{proof}

By the above theorem, a connected algebra $A$ is brick-directed if and only if there is a chain of torsion classes where every $X\in\brick A$ appears as a labeling of the cover relations. This particularly relates to some study of maximal green sequences. For more details, see \cite{KD} and the references therein.

\medskip

Using the above theorem, we obtain the following result, which gives some useful information on the bound quiver of arbitrary brick-directed algebras. In particular, we prove Corollary \ref{Cor: No Cycle}.

\begin{corollary}\label{Cor: brick-directed is weak-triangular}
If $A$ is brick-directed, then $A$ is weakly triangular.
\end{corollary}
\begin{proof}Since $A$ is brick-directed, by Theorem \ref{Thm: Brick-directed & maximal chain of brick-splitting pairs}, there is a maximal chain of brick-splitting torsion pairs in $\tors A$. Thus, we can consider a subchain $$0 = \mathcal{T}_0 \subset \mathcal{T}_1 \subset \mathcal{T}_2 \subset \cdots \subset \mathcal{T}_n = \modu A$$
of the aforementioned maximal chain such that $\mathcal{T}_i$ contains exactly $i$ non-isomorphic simple modules. Now, the result follows from Proposition \ref{prop: chain of brick-splitting separating simples}. 
\end{proof}

Before we give a more combinatorial realization of brick-directed algebras, let us make a remark on weakly triangular algebras and their appearance in our work.

\begin{remark}
We note that weakly triangular algebras have also appeared in the study and classification of \emph{geometrically irreducible algebras} (over algebraically closed fields). These are the algebras for which every module variety $\modu (A,\underline{d})$ is irreducible, where the dimension vector $\underline{d} \in \mathbb{Z}^n_{\geq 0}$ varies. In fact, it is shown that every geometrically irreducible algebras is weakly triangular. For more details and new results, see \cite{BZ} and the references therein. 

Meanwhile, we note that brick-directed algebras and geometrically irreducible algebras are two different subfamilies of weakly triangular algebras.
On the one hand, the path algebra of the $3$-Kronecker quiver is clearly geometrically irreducible, but it is strictly wild, hence it cannot be brick-directed (see Corollary \ref{Cor: strictly wild is never brick-directed}). On the other hand, the local algebra $A=k[x,y]/\langle x^2, y^2, xy\rangle$ is obviously brick-directed, but it is not geometrically irreducible (see \cite[Lemma 2.1]{BZ}). Nevertheless, the study of brick-directed algebras through the geometric lens is an interesting problem for future investigations (also see Remark \ref{Rem: Unique dim vect but not brick-direcetd}).
\end{remark}

\subsection{Basic properties and examples of brick-directed algebras}

We start this subsection with an example of a brick-directed algebra that is not representation-directed.

\begin{example}
Let $A$ be the algebra given in Example \ref{Example: Non-trivial brick-splitting torsion pairs}. More specifically, we have $\brick A=\{{\begin{smallmatrix} 1\end{smallmatrix}}, {\begin{smallmatrix} 2\end{smallmatrix}}, {\begin{smallmatrix} 1\\2\end{smallmatrix}} \}$, and a simple computation shows that 
$$\Hom_A({\begin{smallmatrix} 1\\2\end{smallmatrix}},{\begin{smallmatrix} 2\end{smallmatrix}})=\Hom_A({\begin{smallmatrix} 1\end{smallmatrix}},{\begin{smallmatrix} 1\\2\end{smallmatrix}})=\Hom_A({\begin{smallmatrix} 2\end{smallmatrix}},{\begin{smallmatrix} 1\end{smallmatrix}})=\Hom_A({\begin{smallmatrix} 1\end{smallmatrix}},{\begin{smallmatrix} 2\end{smallmatrix}})=0.$$
Hence, $A$ is brick-directed. We observe that there exists a chain of bricks of $A$ consisting of all bricks. 
In contrast, we note that $A$ is not representation-directed, simply because it admits indecomposables that are not bricks. In fact, one can easily check that every $X \in \brick A$ lies on a cycle in $\modu A$.
\end{example}

\begin{remark}
The previous example shows that brick-directed algebras give a non-trivial generalization of representation-directed algebras. In particular, we note that for a brick-directed algebra $A$, all bricks of $A$ may lie on a cycle in $\modu A$.
\end{remark}

In the following, we give some characterizations of brick-directed algebras and study some of their fundamental properties. 
The following lemma is handy for some of the following arguments. 

\begin{lemma} \label{Lemma: full subcategory}
If $B$ is an algebra such that $\modu B$ is equivalent to a full subcategory of $\modu A$ and $A$ is brick-directed, then $B$ is brick-directed.
\end{lemma}
\begin{proof}
Observe that $\modu B$ cannot have any brick-cycle, as otherwise, by fullness assumption, we obtain a brick-cycle in $\modu A$, hence a contradiction.
\end{proof}

As a consequence of the above lemma, one can show that brick-directedness is preserved under several algebraic operations. For the definition and properties of the $\tau$-tilting reduction (shortly, $\tau$-reduction), we refer to \cite{Ja}.

\begin{proposition}\label{Prop: some properties of brick-directed algebras}
Let $A$ be an algebra. Then, 
\begin{enumerate}
    \item If $A$ is brick-directed, then so is $A/J$, for each ideal $J$ in $A$.
    \item If $A$ is brick-directed, then so is $eAe$, for each idempotent $e\in A$.
    \item If $A$ is brick-directed, then so is every $\tau$-reduction of $A$.
\end{enumerate}
\end{proposition}

\begin{proof}
Parts $(1)$ and $(2)$ directly follow from Lemma \ref{Lemma: full subcategory}. In particular, we observe that the functor $Ae _{eAe}\otimes - $ realizes $\modu eAe$ as a full subcategory of $\modu A$. For part $(3)$, it follows from \cite[Theorem 3.8]{Ja} or \cite[Theorem 4.12 (b)]{DI+} that any $\tau$-tilting reduction $B$ of $A$ is such that $\modu B$ can be realized as a wide subcategory of $\modu A$. In particular, if $A$ is brick-directed, then there is no brick-cycle in any wide subcategory of $\modu A$, hence $\modu B$ has no brick-cycle.
\end{proof}

As another useful consequence of Lemma \ref{Lemma: full subcategory}, we get the following result. For the definition and different characterizations of (strictly) wild algebras, we refer to \cite[Chapter XIX]{SS}.

\begin{corollary}\label{Cor: strictly wild is never brick-directed}
Strictly wild algebras are never brick-directed.
\end{corollary}
\begin{proof}
Let $B$ be the cyclic Nakayama algebra of rank $2$ whose radical square is zero. If $A$ is strictly wild, then there is a full and faithful functor $\mathcal{F}: \modu B \to \modu A$ that respects indecomposables. Take the brick-cycle $P_1 \to P_2 \to P_1$ in $\modu B$. Then $\mathcal{F}(P_1) \to \mathcal{F}(P_2) \to \mathcal{F}(P_1)$ is a brick-cycle in $\modu A$. Thus, $A$ is not brick-directed.
\end{proof}

We now give a full classification of brick-directed hereditary algebras.

\begin{proposition}\label{Prop: Hereditary brick-directed algs}
Let $A$ be a connected path algebra (i.e., $A=kQ$). Then, $A$ is brick-directed if and only if $Q$ is Dynkin or the Kronecker quiver.
\end{proposition}
\begin{proof}
First note that $A=kQ$ is representation-finite if and only if $Q$ is a Dynkin quiver. That being the case, it is well-known that $A=kQ$ is representation-directed and therefore brick-directed. Moreover, if $Q$ is the Kronecker quiver, then $\brick A$ consists of all those indecomposables belonging to the preprojective component and preinjective component of $\Gamma_A$, together with the quasi-simple modules lying on the mouth of the tubes. It is well-known that tubes of $\Gamma_A$ are homogeneous and pairwise hom-orthogonal, that is, there is no brick-cycle consisting of those bricks lying on the tubes. Moreover, morphisms between indecomposable modules in $\modu A$ are directed, in particular, from the preinjective modules there is no non-zero homomorphism to the regular modules on the tubes nor to the preprojective modules. Moreover, there is no non-zero homomorphism from the regular modules on the tubes to the preprojective modules. Hence, for the Kronecker quiver $Q$, the path algebra $kQ$ is brick-directed. 

Suppose $Q$ is an extended Dynkin quiver different from the Kronecker quiver. We show that $A=kQ$ is not brick-directed. In particular, we know that $A=kQ$ is tame and $\Gamma_A$ has at least one stable tube of rank greater than $1$. In particular, there exists a tube $\mathbf{t}$ in $\Gamma_A$ such that the first two layers of $\mathbf{t}$ consist of bricks. Thus, one can easily find a brick-cycle in $\mathbf{t}$. This shows that $A=kQ$ is not brick-directed. 

Finally, if $Q$ is a quiver that is neither Dynkin nor extended Dynkin, the path algebra $kQ$ is strictly wild, and hence Corollary \ref{Cor: strictly wild is never brick-directed} implies that $A=kQ$ is not brick-directed. This finishes the proof.
\end{proof}

\subsection{Construction of new brick-directed algebras}\label{Subsection: Construction of new brick-directed algebras}
We now give an explicit construction of new examples of brick-directed algebras via a process known as ``gluing" two algebras. In particular, we closely follow the construction introduced in \cite[Section 7]{MP1}. This construction allows us to start from a pair of brick-directed algebras whose quivers satisfy some properties and construct a large family of new examples of brick-directed algebras.

\medskip

Let $A=kQ/I$ and $B=kQ'/I'$ be two connected algebras, where $I$ and $I'$ are assumed to be admissible ideals. Let $R$ and $R'$ denote two sets of admissible relations, respectively, in $kQ$ and $kQ'$, which generate $I$ and $I'$. 
Moreover, assume that $Q_0=\{x_1, \ldots, x_n\}$ and $Q'_0=\{y_1, \ldots, y_m\}$ denote, respectively, the vertex set of $A$ and $B$.
Provided that $x_n$ is a sink in $Q$, and $y_1$ is a source in $Q'$, then by $\Lambda:=k\Tilde{Q}/\Tilde{I}$ we denote the algebra obtained from \emph{gluing} of $A$ and $B$ at the vertices $x_n$ and $y_1$, with the bound quiver $(\Tilde{Q}, \Tilde{I})$ constructed as follows: 
\medskip

$\mathbf{\Tilde{Q}}$: The quiver obtained from $Q$ and $Q'$, by identifying the vertices $x_n$ and $y_1$, and keeping the rest of $Q$ and $Q'$ untouched. Hence, the set of vertices of $\Tilde{Q}$ is $\Tilde{Q}_0=(Q_0-\{x_n\})\cupdot(Q'_0-\{y_1\}) \cupdot \{v\}$, where vertex $v$ is obtained from the identification of $x_n$ and $y_1$. Thus, we obtain $\Tilde{Q}_0=\{x_1,\dots,x_{n-1},v,y_2,\dots,y_m\}$, and the set of arrows is given by $\Tilde{Q}$ is $\Tilde{Q}_1= Q_1\cupdot Q'_1$.

\medskip
$\mathbf{\Tilde{I}}$: The admissible ideal in $k\Tilde{Q}$ is generated by the set $\Tilde{R}:= R \cupdot R' \cupdot R_v$. Here, $R_v$ denotes the set of all the paths of length $2$ in $\Tilde{Q}$ that pass through vertex $v$. More specifically, each element of $R_v$ is of the form $\beta\alpha$, for an arrow $\alpha \in Q_1$ ending at vertex $x_n$, and an arrow $\beta \in Q'_1$ starting at vertex $y_1$.

\medskip
From the construction, it is immediate that $(\Tilde{Q}, \Tilde{I})$ contains a copy of $(Q,I)$, as well as a copy of $(Q',I')$, as bound subquivers.
Moreover, other than vertex $v$ in $\Tilde{Q}$, any other vertex in $\Tilde{Q}_0$ belongs either to $Q_0$ or $Q'_0$, and not both. 
We note that because $B$ has a source, it is immediate that $\Lambda=k\Tilde{Q}/\Tilde{I}$ is isomorphic to $A=kQ/I$ if and only if the bound quiver of $B$ consists of only one vertex, that is, if $B\simeq k$.

\medskip

Some of the properties of the gluing process are listed in the following proposition. 

\begin{proposition}\label{Prop:on gluing}
Let $A=kQ/I$ and $B=kQ'/I'$ be two connected algebras. If $\Lambda=k\Tilde{Q}/\Tilde{I}$ is obtained from gluing $A$ and $B$, then $\Lambda$ is a connected algebra of rank $|A|+|B|-1$. Furthermore, the following hold:

\begin{enumerate}
\item If $M$ is indecomposable, then $M$ is either entirely supported over $Q$ or entirely supported over $Q'$.
\item $\Lambda$ is brick-directed if and only if both $A$ and $B$ are brick-directed.
\end{enumerate}
\end{proposition}
\begin{proof}
The first statement is clear. For the second statement, consider the brick quivers $Q^b(A)$ and $Q^b(B)$ of $A$ and $B$, respectively. Let $S_v$ be the simple module of $\Lambda$ corresponding to the vertex $v$ obtained from gluing the sink $x_n$ in $A$ with the source $y_1$ in $B$. 
We first observe that $S_v$ is a source in $Q^b(A)$ while it is a sink in $Q^b(B)$. Statement (2) follows from the fact that the brick quiver $Q^b(\Lambda)$ of $\Lambda$ is nothing but the gluing of $Q^b(A)$ and $Q^b(B)$ at $S_v$.
\end{proof}

Using the construction and proposition above, in the rest of this subsection we develop some tools to prove Corollary \ref{Cor on various families of brick-directed algebras}.  For that, we begin with some useful observations on some explicit examples of brick-directed algebras of smaller ranks that we use to construct brick-directed algebras of higher ranks. Meanwhile, let us recall that, by Drozd's trichotomy theorem, each algebra $A$ falls into exactly one of the three disjoint subfamilies, given by representation-finite, (representation-infinite) tame or wild algebras. For the definitions and further details on these subfamilies of algebras and the trichotomy theorem, see \cite[Chapter XIX]{SS}.

\medskip

First, consider the local algebra $L_n= k[x_1,\cdots, x_n]/\langle x_ix_j : 1\leq i\leq j \leq n \rangle$. Because each local algebra is of rank $1$, the bound quiver of $L_n$ consists of $n$ loops that start and end at the same vertex, subject to all quadratic relations. It is well-known that $L_1$ is representation-finite, $L_2$ is (representation-infinite) tame, and $L_n$ is wild for all $n>2$. Observe that $\brick L_n$ consists of only the simple module associated to the unique vertex of the quiver of $L_n$. Hence, $L_n$ is a trivially brick-finite and brick-directed.

\medskip

Next, we give a concrete family of brick-finite algebras of rank $n>1$ which are (representation-infinite) tame and brick-directed.

\begin{example}\label{Example: windwheel alg of rank 2}
For a positive integer $n>1$, let $Q_n$ denote the quiver
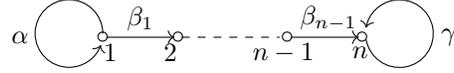
\begin{center}
    \begin{tikzpicture}
\node at (2.5, 2.5) {$\circ$}; 
\node at (2.6, 2.3) {$1$}; 
    \draw[->] (2.5,2.55) arc (0:340:0.45cm);
\node at (1.4,2.5) {$\alpha$};
        \node at (3.5,2.5) {$\circ$};
        \node at (3.4,2.3) {$2$};
   \draw [->] (2.55,2.5) --(3.45,2.5);
   \node at (3,2.75) {$\beta_1$};
        \node at (4.95,2.5) {$\circ$};
        \node at (4.9,2.3) {$n-1$};
    \draw [dashed] (3.55,2.5) --(4.85,2.5); 
        \draw [->] (5,2.5) --(5.9,2.5) ; 
       \node at (5.5,2.75) {$\beta_{n-1}$};
        \node at (5.95,2.5) {$\circ$};
        \node at (5.92,2.3) {$n$};
    \draw[->] (6,2.47) arc (190:530:0.45cm);
\node at (7.1,2.5) {$\gamma$};
    \end{tikzpicture}
\end{center}
where all arrows $\beta_i$, for $1\leq i \leq n-1$, are equi-oriented, that is, $\beta_i: i\rightarrow i+1$. In particular, the bar $\mathbf{b}:=\beta_{n-1}\cdots \beta_1$, which starts at $1$ and ends at $n$, is a copy of the equi-oriented Dynkin quiver $A_{n}$.

Set $\Lambda_{n}:=kQ_n/I_n$, with $I_n:=\langle \alpha^2, \gamma^2, \gamma\beta_{n-1}\cdots \beta_1\alpha \rangle$. Then, $\Lambda_{n}$ is a representation-infinite string algebra. In fact, $\Lambda_{n}$ is a minimal representation-infinite algebra of domestic type with a unique band $\mathbf{b}^{-1}\gamma^{-1}\mathbf{b}\alpha$, and $\Lambda_{n}$ is brick-finite (see Remark \ref{Rem: windwheel algebras}).

To list all elements of $\brick \Lambda_{n}$, we can use some standard combinatorial arguments over string algebras (for details, see \cite[Chapter III]{Mo1} and references therein). More precisely, for $X\in \brick \Lambda_{n}$, either $X$ is a simple module, or else $X$ is given by one of the following strings:

\begin{itemize}
    \item $w_j:=\beta_{j}\cdots \beta_1\alpha$, for $1<j\leq n$;
    \item $u_{ij}:=\beta_{j}\cdots \beta_i$, for $1\leq i < j \leq n$;
    \item $w_i:=\gamma\beta_{n-1}\cdots \beta_i$, for $1\leq i<n$;
\end{itemize}
Furthermore, thanks to the combinatorial description of morphisms between string modules, one can easily draw the brick quiver of $\Lambda_m$ and observe that there is no cycle in $Q^b(\Lambda_n)$. Thus, $\Lambda_n$ is a tame algebra that is brick-finite and brick-directed.

As an explicit computation of the above descriptions, for $n=2$, we have
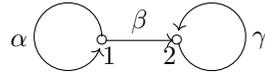
\begin{center}
    \begin{tikzpicture}
    \draw[->] (2.5,2.55) arc (0:340:0.45cm);
\node at (1.4,2.5) {$\alpha$};
        \node at (3.5,2.5) {$\circ$};
        \node at (3.4,2.3) {$2$};
   \draw [->] (2.55,2.5) --(3.45,2.5) ; 
\node at (3,2.75) {$\beta$};
    \node at (2.5, 2.5) {$\circ$}; 
    \node at (2.6, 2.3) {$1$}; 
    \draw[->] (3.52,2.47) arc (190:530:0.45cm);
\node at (4.6,2.5) {$\gamma$};
    \end{tikzpicture}
\end{center}
subject to the relations $\alpha^2=\gamma^2=\gamma\beta\alpha=0$. 
Here, the unique band in $(Q_2,I_2)$ is $\beta^{-1}\gamma^{-1}\beta\alpha$, and bricks in $\modu \Lambda_2$ are given by the set of strings $\{e_1, e_2, \beta, \beta\alpha,\gamma\beta\}$. Here, $e_1$ and $e_2$ respectively denote the string of zero length corresponding to the simple modules $S_1$ and $S_2$. 
Hence, expressed by their radical filtration, we have
$$\brick \Lambda_2=\left\{{\begin{smallmatrix} 1\end{smallmatrix}}, {\begin{smallmatrix} 2\end{smallmatrix}}, {\begin{smallmatrix} 1\\2\end{smallmatrix}}, {\begin{smallmatrix} 1\\1\\2\end{smallmatrix}}, 
{\begin{smallmatrix} 1\\2\\2\end{smallmatrix}}\right\},$$ 
and the brick quiver $Q^b(\Lambda_2)$ is 
$$
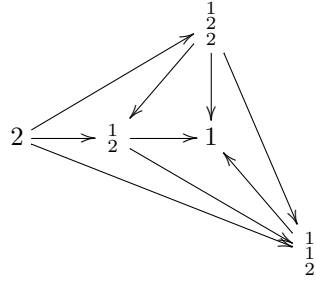
\xymatrix{ &  & {\begin{smallmatrix} 1\\2\\2\end{smallmatrix}} \ar[d] \ar[dl] \ar[ddr]& \\
2 \ar[urr] \ar[r] \ar[drrr] & {\begin{smallmatrix} 1\\2\end{smallmatrix}} \ar[r] \ar[drr]  & 1 & \\
& & & {\begin{smallmatrix} 1\\1\\2\end{smallmatrix}} \ar[ul]}$$
which has no oriented cycle. Hence, $\Lambda_2$ is a (representation-infinite) string algebra of rank $2$ which is brick-finite and brick-directed.
\end{example}

\begin{remark}\label{Rem: windwheel algebras}
The algebra $\Lambda_n$ in Example \ref{Example: windwheel alg of rank 2} belongs to a particular family of minimal representation-infinite special biserial algebras, so-called ``wind-wheel" algebras, introduced in \cite{Ri2}. As shown in \cite[Prop. 3.1.10]{Mo2}, wind-wheel algebras are brick-finite. 
We further note that, by Proposition \ref{Prop: some properties of brick-directed algebras}, and due to the minimality, each proper quotient of $\Lambda_n$ is a representation-finite brick-directed algebra. 
For instance, the algebra $A$ in Example \ref{Example: Non-trivial brick-splitting torsion pairs} is a proper quotient of $\Lambda_2$ treated in Example \ref{Example: windwheel alg of rank 2}. 

We also use these concrete examples to highlight another fundamental difference between representation-directed algebras and brick-directed algebras. First, we recall that if $A$ is a connected representation-finite algebra, then $A$ is representation-directed if and only if the (unique) component of $\Gamma_A$ is preprojective (see \cite[IX.3, Lemma 3.5]{ASS}). That is, among connected algebras, representation-directed algebras can be classified as those whose Auslander-Reiten quiver consists of a unique preprojective component. However, brick-directed algebras may or may not admit a preprojective component. For instance, the Kronecker algebra is brick-directed (see Prop. \ref{Prop: Hereditary brick-directed algs}) and it is known to admit a (unique infinite) preprojective component. In contrast, the wind-wheel algebras in Example \ref{Example: windwheel alg of rank 2}, which are shown to be brick-directed, do not admit any preprojective components. This is because they are representation-infinite but brick-finite (also see \cite{Mo1}). Also, Example \ref{Example: Non-trivial brick-splitting torsion pairs} gives an explicit example of a connected representation-finite algebra which is brick-directed without a preprojective component.
\end{remark}

Now we prove Corollary \ref{Cor on various families of brick-directed algebras}, which we also state below. Consequently, we obtain an abundance of brick-directed algebras of different types.

\begin{corollary}\label{Various families of brick-directed algebras}
For any integer $n>1$, there exists a connected brick-directed algebra $\Lambda$ of rank $n$ of any of the following types:
\begin{itemize}
    \item[(I)] representation-finite algebra;
    \item[(II)] brick-finite tame algebra;
    \item[(III)] brick-infinite tame algebra;
    \item[(IV)] brick-finite wild algebra;
    \item[(V)] brick-infinite wild algebra.
\end{itemize}
\end{corollary}
\begin{proof}
Let $n>1$ be a fixed integer. For every case in the assertion, we give an explicit algebra $\Lambda$ of rank $n$ which satisfies the desired properties.

For (I), let $Q$ denote a Dynkin quiver of rank $n$. Note that $kQ$ is representation-directed, hence it is a representation-finite and brick-directed algebra. 

For (II), consider the algebra $\Lambda_n$ given in Example \ref{Example: windwheel alg of rank 2}.

For (III), first recall that for the Kronecker quiver $Q: 
1\doublerightarrow{\,}{\,}2$, the path algebra $A=kQ$ is tame and brick-directed (see Proposition \ref{Prop: Hereditary brick-directed algs}). We also consider the equi-oriented quiver $Q': 1\xrightarrow{\beta_1} 2\xrightarrow{\beta_2} \cdots n-2 \xrightarrow{\beta_{n-2}} n-1$ and let $B=kQ'$. Note that $B$ is a (representation-finite) brick-directed algebra. By Proposition \ref{Prop:on gluing}, via gluing construction that identifies vertex $2$ in $Q$ with vertex $1$ in $Q'$, we obtain the tame brick-directed algebra $\Lambda=k\Tilde{Q}/\Tilde{I}$, which is of rank $n$. Since $A$ is a proper quotient of $\Lambda$, then $\Lambda$ is obviously brick-infinite.

For (IV), let $A$ be the wild local algebra $k[x_1,x_2, x_3]/\langle x_ix_j : 1\leq i\leq j \leq 3 \rangle$, whose bound quiver $(Q,I)$ consists of a single vertex and $3$ loops, say $\alpha_1, \alpha_2$ and $\alpha_3$, subject to all quadratic relations. Since $\brick A$ consists only of the unique simple $A$-module, $A$ is evidently brick-finite and brick-directed. Moreover, let $Q'$ be the equi-oriented Dynkin quiver $1\xrightarrow{\beta_1} 2\xrightarrow{\beta_2} \cdots n-1 \xrightarrow{\beta_{n-1}} n$. As remarked previously, $B=kQ'$ is a (representation-finite) brick-directed algebra. 
Now, we construct the desired algebra $\Lambda=k\Tilde{Q}/\Tilde{I}$ from $A$ and $B$, similar to the gluing construction described before Proposition \ref{Prop:on gluing}. More precisely, identify vertex $n$ in $Q'$ with the unique vertex of $Q$ to obtain the new quiver $\Tilde{Q}$, and let $\Tilde{I}:=\langle \alpha_i\alpha_j, \alpha_i\beta_{n-1}: 1\leq i\leq j \leq 3 \rangle$. 
Observe that $A$ is a proper quotient of $\Lambda$, hence $\Lambda$ is wild. Observe that if $X$ is a brick over $\Lambda$, then either $X(\beta_{n-1})=0$ or $X(\alpha_{i})=0$ for all $1 \le i \le 3$. Otherwise, the simple module $S_n$ appears both in the socle and the top of $X$, so $X$ cannot be a brick. Consequently, $\brick \Lambda=\brick B$, so $\Lambda$ is brick-finite and brick-directed. 

For (V), the construction of the algebra $\Lambda$ is similar to the previous case. In particular, let $A$ be the wild local algebra given in part $(IV)$. However, here we take $Q'$ as the equi-oriented quiver $1\doublerightarrow{\delta_1}{\delta_2}{\,}{\,}2 \xrightarrow{\beta_2} \cdots \xrightarrow{\beta_{n-2}} n-1 \xrightarrow{\beta_{n-1}} n$, and define the algebra $B:=kQ'/I'$, where $I':=\langle \beta_2\delta_1, \beta_2\delta \rangle$. We observe that, by Proposition \ref{Prop:on gluing}, $B$ is a tame brick-infinite brick-directed algebra. 
Now, we construct the desired algebra $\Lambda=k\Tilde{Q}/\Tilde{I}$ from $A$ and $B$, similar to the previous case. Namely, we first identify vertex $n$ in $Q'$ with the unique vertex of $Q$ to obtain the new quiver $\Tilde{Q}$. Then, consider the algebra $\Lambda=k\Tilde{Q}/\Tilde{I}$, where $\Tilde{I}:=\langle \beta_2\delta_1, \beta_2\delta_2, \alpha_i\alpha_j, \alpha_i\beta_{n-1}: 1\leq i\leq j \leq 3 \rangle$. 
Since $A$ and $B$ are proper quotients of $\Lambda$, we evidently have that $\Lambda$ is wild and brick-infinite. Furthermore, as observed in the previous case, we have $\brick \Lambda=\brick B$, which immediately implies that $\Lambda$ is brick-directed.
\end{proof}

\subsection{Wall-chamber structure of brick-directed algebras}\label{Subsection: Wall-chamber structure of brick-directed algebras} 
Thanks to the rich connections between the study of bricks and the stability conditions, we can see some important properties of the brick-directed algebras through their $g$-vector fans. In particular, as the main result of this subsection, for brick-finite algebras we give an interesting characterization of brick-directed algebras in terms of their wall-and-chamber structures. For all the background, motivations, and undefined terminology on the $g$-vector fans and connections to stability conditions, we refer to \cite{As-wc, AH+, BKT,BST} and the references therein.

\medskip

As before, let $A$ be an algebra of rank $n$, and consider the real Grothendieck group $K_0(\proj A)_\mathbb{R}=\bigoplus_{i=1}^n \mathbb{R}[P_i]$. 
We consider the \emph{Euler bilinear form} 
$$K_0(\proj A)_\mathbb{R} \times K_0(\modu A)_\mathbb{R} \to \mathbb{R}; \quad
([P_i],[S_j]) \mapsto \delta_{i,j} \dim_k \End_A(S_j).$$
Via this bilinear form, each $\theta \in K_0(\proj A)_\mathbb{R}$ induces an $\mathbb{R}$-linear form $K_0(\modu A)_\mathbb{R} \to \mathbb{R}$.
As in \cite{Ki}, for $\theta \in K_0(\proj A)_\mathbb{R}$, a module $M \in \modu A$ is said to be \emph{$\theta$-semistable} if $\theta([M])=0$ and $\theta([N]) \ge 0$ for every quotient module $N$ of $M$. If a $\theta$-semistable module $M$ moreover satisfies $\theta([N])>0$ for every non-zero proper quotient module $N$ of $M$, then $M$ is said to be \emph{$\theta$-stable}. Following \cite{BST}, for each module $M \in \modu A$, we define
$$\Theta_M := \{\theta \in K_0(\proj A)_\mathbb{R} \mid \text{$M$ is $\theta$-semistable}\}.$$
Recall that a \emph{wall} in $K_0(\proj A)_\mathbb{R}\simeq \mathbb{R}^n$ is a set of the form $\Theta_X$, where $X \in \brick A$ and $\Theta_X$ has co-dimension one in $\mathbb{R}^n$. 
Every wall $\Theta_X$ has a unique supporting hyperplane defined as $\mathcal{H}_X := \{\theta \in K_0(\proj A)_\mathbb{R} \mid \theta([X])=0\}$. Observe that the converse is not true in general, that is, two distinct walls may have the same supporting hyperplane. 

For each $\theta \in K_0(\proj A)_\mathbb{R}$, \cite{BKT} associated two torsion pairs $(\overline{\mathcal{T}}_\theta,\mathcal{F}_\theta)$ and $(\mathcal{T}_\theta,\overline{\mathcal{F}}_\theta)$ in $\modu A$ by
\begin{align*}
\overline{\mathcal{T}}_\theta&:=\{M \in \modu A \mid \text{$\theta([N]) \ge 0$ for every quotient module $N$ of $M$}\},\\
\mathcal{F}_\theta&:=\{M \in \modu A \mid \text{$\theta([L]) <0$ for every submodule $L \ne 0$ of $M$}\},\\
\mathcal{T}_\theta&:=\{M \in \modu A \mid \text{$\theta([N]) >0$ for every quotient module $N \ne 0$ of $M$}\}, \\
\overline{\mathcal{F}}_\theta&:=\{M \in \modu A \mid \text{$\theta([L]) \ge 0$ for every submodule $L$ of $M$}\}.
\end{align*}
We always have $\mathcal{T}_\theta \subseteq \overline{\mathcal{T}}_\theta$ and $\mathcal{F}_\theta \subseteq \overline{\mathcal{F}}_\theta$. The difference between $(\overline{\mathcal{T}}_\theta,\mathcal{F}_\theta)$ and $(\mathcal{T}_\theta,\overline{\mathcal{F}}_\theta)$ are expressed by the subcategory $\mathcal{W}_\theta:=\overline{\mathcal{T}}_\theta \cap \overline{\mathcal{F}}_\theta$, which is the subcategory consisting of all $\theta$-semistable modules. The subcategory $\mathcal{W}_\theta$ is known to be a wide subcategory of $\modu A$ whose simple objects are exactly the $\theta$-stable modules. The following two results on brick-finite algebras are known and play important roles in our proofs below. For the details, we refer to \cite{As-smc}. 

 \begin{prop}\label{Prop: functorially finite torsion pairs are semistable}
     Let $A$ be brick-finite. If $\mathcal{T} \to \mathcal{U}$ is an arrow in $\Hasse(\tors A)$ labeled by a brick $X$, then there exists $\theta \in K_0(\proj A)_\mathbb{R}$ such that $\overline{\mathcal{T}}_\theta=\mathcal{T}$, $\mathcal{T}_\theta=\mathcal{U}$ and that $X$ is $\theta$-stable.
 \end{prop}

 The following lemma collects some key properties of the wall and chamber structure of a brick-finite algebra.

\begin{lemma} \label{Lemma: relative-interior of walls}
    Let $A$ be brick-finite. 
    \begin{enumerate}
    \item 
    For each $X \in \brick A$, the set $\Theta_X$ is of co-dimension one; that is, $\Theta_X$ is a wall.
    \item 
    Let $\theta \in K_0(\proj A)_\mathbb{R}$. Then $\mathcal{W}_\theta$ has only finitely many simple objects $X_1,\ldots,X_m$ up to isomorphisms, and their dimension vectors in $K_0(\modu A)$ are extendable to a $\mathbb{Z}$-basis of $K_0(\modu A)$.
    \item
    If $X$ and $Y$ are non-isomorphic bricks, then the co-dimension of $\Theta_X \cap \Theta_Y$ in $K_0(\proj A)_\mathbb{R}$ is at least two.
    \end{enumerate}
\end{lemma}

\begin{remark}
Note that the assumption that $A$ is brick-finite is crucial in Lemma \ref{Lemma: relative-interior of walls}. To stress this observation, we provide an example showing that $(3)$ does not hold, in general, for a brick-infinite algebra. Consider the string algebra $A=kQ/I$, where $Q$ is given by
    $$
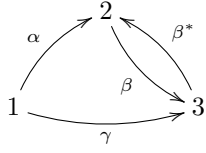
\xymatrix{& 2 \ar@/_/[dr]_\beta & \\
1 \ar@/_/[rr]_\gamma \ar@/^/[ur]^\alpha & & 3 \ar@/_/[ul]_{\beta^*}}$$
and $I = \langle \beta\beta^*, \beta^*\beta\rangle$. In particular, $A$ is a finite-dimensional quotient algebra of the preprojective algebra of type $\tilde{\mathbb{A}}_2$. Consider the uniserial string modules $X=\begin{smallmatrix}1\\2\\3\end{smallmatrix}$ and $Y=\begin{smallmatrix}1\\3\\2\end{smallmatrix}$. Observe that if $\theta = a[P_1] + b[P_2]+c[P_3]$ with $a > 0$, $b,c < 0$ and $a+b+c=0$, then we see from the definition that both $X,Y$ are $\theta$-stable. Hence, we see that the codimension of $\Theta_X \cap \Theta_Y$ in $K_0(\proj A)_\mathbb{R}$ is one. Observe that $X$ is $([P_1] - [P_3])$-stable but not $([P_1]-[P_2])$-stable, while $Y$ is $([P_1]-[P_2])$-stable but not $([P_1]-[P_3])$-stable, showing that $\Theta_X \ne \Theta_Y$. 
\end{remark}

\medskip

In order to state the main result of this part more succinctly, we introduce a new terminology. In particular, for a totally ordered set $J$, a sequence $\{\theta_j\}_{j\in J}$ of elements of $K_0(\proj A)_{\mathbb{R}}$ is called \emph{consistent} (respectively, \emph{weakly consistent}) if there exists an ordering $\brick A=\{X_j\}_{j\in J}$ on $\brick A$ such that

 \begin{enumerate}
     \item For $i < j$, we have $\theta_i(X_j) < 0$;
     \item For $i > j$, we have $\theta_i(X_j) > 0$ (respectively, $\theta_i(X_j) \geq 0$);
     \item The brick $X_i$ is $\theta_i$-stable (in particular, $\theta_i(X_i) = 0$).
 \end{enumerate}

\begin{remark}
   Let $\{\theta_j\}_{j\in J}$ be a consistent sequence as above. For each $j$, consider the wall $\Theta_{X_j}$. For each $\Theta_{X_j}$, we can find a $g$-vector $\theta_j$ in the relative-interior of the wall $\Theta_{X_j}$ such that the negative side of the hyperplane $\mathcal{H}_{X_j}$ contains all $\theta_i$ with $i < j$, and its positive side contains all $\theta_i$ with $i > j$. 
\end{remark}
 
Now, we show the main result of this subsection. 

 \begin{theorem}\label{Prop: walls ordered consistently}
     Let $A$ be brick-finite. Then $A$ is brick-directed if and only if $K_0(\proj A)_{\mathbb{R}}$ admits a consistent sequence if and only if $K_0(\proj A)_{\mathbb{R}}$ admits a weakly consistent sequence.    
 \end{theorem}

 \begin{proof}
     We prove that if $A$ is brick-directed, then $K_0(\proj A)_{\mathbb{R}}$ admits a consistent sequence. By Theorem \ref{Thm: Brick-directed & maximal chain of brick-splitting pairs}, there exists a maximal chain $\mathcal{T}_1\subset \mathcal{T}_2\subset \cdots\subset \mathcal{T}_r$ of brick-splitting torsion classes. Thus we have an ordering  $\brick A=\{X_1, X_2,\ldots, X_r\}$ on $\brick A$ such that
     \[\mathcal{T}_i\cap\brick A=\{X_j\mid 1\le j\le i\}\ \mbox{ and }\ \mathcal{F}_i\cap\brick A=\{X_j\mid i<j\le r\},\] 
     where $\mathcal{F}_i:=\mathcal{T}_i^\perp$. Since $A$ is brick-finite, Proposition \ref{Prop: functorially finite torsion pairs are semistable} shows that there exists $\theta_i\in K_0(\proj A)_\mathbb{R}$ such that $\mathcal{T}_{\theta_i}=\mathcal{T}_{i-1}$ and $\overline{\mathcal{T}}_{\theta_i}=\mathcal{T}_i$. Then $\theta_1, \theta_2, \ldots,\theta_r$ satisfy the desired condition. 
     In fact, $X_i\in\mathcal{T}_i\cap\mathcal{F}_{i-1}=\overline{\mathcal{T}}_{\theta_i}\cap\overline{\mathcal{F}}_{\theta_i}=\mathcal{W}_{\theta_i}$ implies condition (3).
     For $i<j$, $X_j\in\mathcal{F}_i=\mathcal{F}_{\theta_i}$ implies condition (1). For $i>j$, $X_j\in\mathcal{T}_{i-1}=\mathcal{T}_{\theta_i}$ implies condition (2).

     We now prove that, if $K_0(\proj A)_{\mathbb{R}}$ admits a weakly consistent sequence, then $A$ is brick-directed.
    Assume that with the total order $1 < 2 < \cdots < r$, we have been able to choose a weakly consistent sequence $\{\theta_j\}_{j \in J}$ with respect to an ordering $\brick A=\{X_j\}_{j \in J}$. For each $i$, we let $\mathcal{T}_i := \overline{\mathcal{T}}_{\theta_i}$. 
     By induction on $i$, we claim that
     \[\mathcal{T}_i = \Filt(X_1 \oplus \cdots \oplus X_i).\]
     Since $\mathcal{T}_i=\Filt(\mathcal{T}_i \cap \brick A)$ holds by \cite[Lemma 3.9]{DI+}, it suffices to show $\mathcal{T}_i \cap \brick A=\{X_1,\ldots,X_i\}$. The inclusion ``$\subseteq$'' follows by (1). We prove ``$\supseteq$''.
     By (3), we have $X_i \in \mathcal{T}_i$. 
     For each $j<i$, take any quotient module $N$ of $X_j\in\mathcal{T}_j$. Since $N\in\mathcal{T}_j=\Filt(X_1\oplus\cdots\oplus X_j)$ by the induction hypothesis, we have $\theta_i([N])\ge0$ by (2). Thus $X_j\in\mathcal{T}_i$, as desired. We have proved ``$\supseteq$''.     

     The constructed chain $0 \subset \mathcal{T}_1 \subset \cdots \subset \mathcal{T}_r = \modu A$ yields that $A$ is brick-directed.
 \end{proof}

\section{Left modular lattices in representation theory}\label{Section: Trim lattices in representation theory}
In this section, we use our results on brick-splitting torsion pairs and brick-directed algebras to give a new characterization of several lattice theoretical notions in the setting of lattice of torsion classes. In particular, we prove Theorem \ref{Thm: brick-directed, extremal, left modular, trimness} and Corollaries \ref{Cor: extremailty preserved}, \ref{Cor: Characterization of rep-directed algebras} and \ref{Cor: Uniqueness of dim vector of bricks} from Section \ref{Sec: Summary of Main Results}. More specifically, we fully determine for which algebras, the lattice of torsion classes are extremal, left modular and trim (for definitions, see Section \ref{Section: Preliminaries and Background}).

As remarked before, in the literature of lattice theory, the lattice theoretical notions considered in this work (i.e., extremality, left modularity and trimness) are only defined for finite lattices. Since we study such properties for the lattice of torsion classes, unless specified otherwise, throughout the entire section we only work with those algebras $A$ for which $\tors A$ is a finite lattice. This is known to be the case if and only if $A$ is brick-finite (for details, see \cite{DIJ} and references therein).

In what follows, we first prove Theorem \ref{Thm: brick-directed, extremal, left modular, trimness}. This is obtained as a consequence of the next two propositions, which can be seen as the algebraic counterparts of Theorem \ref{Thm: extremal & semidistributive is trim}, where we are concerned with the lattice of torsion classes of brick-finite algebras.

\begin{proposition}\label{Prop: brick-directed equiv. left modular}
If $A$ is a brick-finite algebra, then the following are equivalent:
\begin{enumerate}
    \item $A$ is brick-directed;
    \item $\tors A$ is left modular.
\end{enumerate}
\end{proposition}
\begin{proof}
By Theorem \ref{Thm: Brick-directed & maximal chain of brick-splitting pairs}, $A$ is brick-directed if and only if $\tors A$ admits a maximal chain of brick-splitting torsion classes. Moreover, Proposition \ref{Prop: brick-splitting is left modular} asserts that brick-splitting torsion classes in $\modu A$ are exactly the left modular elements in $\tors A$. Hence, $A$ is brick-directed if and only if $\tors A$ admits a maximal chain consisting of left modular elements. Therefore, by definition, this implies that $A$ is brick-directed if and only if $\tors A$ is left modular.
\end{proof}

In the next proposition, we consider the notion of extremality of $\tors A$. Before we state the following result, we note that the next proposition can be alternatively concluded from Proposition \ref{Prop: brick-directed equiv. left modular}, together with some recent results of \cite{Mu} on semidistributive left modular lattices (for more details, see Section \ref{Subsection: Lattice theory}). However, we use our results on brick-splitting torsion classes and brick-directed algebras to give a direct proof of this proposition.

\begin{proposition}\label{Prop: brick-directed equiv. extremal}
If $A$ is a brick-finite algebra, the following are equivalent:
\begin{enumerate}
    \item $A$ is brick-directed;
    \item $\tors A$ is extremal.
\end{enumerate}
\end{proposition}
\begin{proof}
First, note that, by Theorem \ref{Thm: bijection between Join-irr and brick-labels}, there is a bijection between $\brick A$ and the set of join-irreducible (similarly, the set of meet-irreducible) elements of $\tors A$ (also see \cite[Theorem 1.5]{BCZ}). Thus, to prove the proposition, we show that $A$ is brick-directed if and only if $\tors A$ admits a maximal chain of length $|{\brick A}|$. 

To prove $(1)\to (2)$, we first note that, by Theorem \ref{Thm: Brick-directed & maximal chain of brick-splitting pairs}, $\tors A$ admits a maximal chain $\{\mathcal{T}_i\}_{i\in I}$ of brick-splitting torsion classes. Moreover, by Proposition \ref{Prop: brick-splitting and intervals}, for each $\mathcal{T}_i$ in this chain, every $X\in \brick A$ belongs to $\brick[0,\mathcal{T}_i]$ or $\brick[\mathcal{T}_i,\modu A]$. Hence, each $X\in \brick A$ appears as a label of the maximal chain $\{\mathcal{T}_i\}_{i\in I}$, implying that $\tors A$ has a maximal chain of length $|{\brick A}|$, thus $\tors A$ is a extremal lattice.

To show $(2)\to (1)$, note that if $\tors A$ is extremal, then $\tors A$ admits a maximal chain of length $|{\brick A}|$. Then, by Theorem \ref{Thm: Demonet's correspondence}, there is a chain of bricks of length $|{\brick A}|$, that is, there is a total order on the set of $\brick A$. This, by definition, implies that there is no brick-cycle in $\modu A$, hence $A$ is brick-directed.
\end{proof}

As a consequence of the preceding propositions, we obtain the following result.

\begin{corollary}\label{Cor: brick-directed equiv. trim}
Let $A$ be brick-finite algebra. Then, $A$ is brick-directed if and only if $\tors A$ is a trim lattice. That being the case, the spine of $\tors A$ forms a distributive sublattice of $\tors A$ which consists of brick-splitting torsion classes.
\end{corollary}
\begin{proof}
By definition, $\tors A$ is a trim lattice if and only if $\tors A$ is extremal and left modular. Hence, the first assertion follows from Propositions \ref{Prop: brick-directed equiv. left modular}
and \ref{Prop: brick-directed equiv. extremal}. Moreover, from \cite[Theorem 3.7]{TW}, an element of a trim lattice is left modular if and only if it lies on the spine. 
Hence, if $\tors A$ is trim, then by Proposition \ref{Prop: brick-splitting is left modular}, a torsion class $\mathcal{T}$ in $\modu A$ belongs to the spine of $\tors A$ if and only if $\mathcal{T}$ is brick-splitting. Moreover, from \cite[Lemma 7]{Th}, it follows that the spine of every trim lattice is a distributive sublattice.
\end{proof} 

From some standard algebraic arguments, together with Proposition \ref{Prop: some properties of brick-directed algebras}, we obtain Corollary \ref{Cor: extremailty preserved}, which we prove below.

\begin{corollary}\label{Cor: quotien, corner and tau-reduction of extremal lattices}
Let $A$ be a brick-finite algebra. If $\tors A$ is extremal  (equivalently, left modular), then the following lattices are also extremal:
\begin{enumerate}
    \item $\tors A/J$, for any 2-sided ideal $J$ in $A$;
    \item $\tors eAe$, for any idempotent element $e$ in $A$;
    \item $\tors B$, for any algebra $B$ which is a $\tau$-reduction of $A$.
\end{enumerate}
\end{corollary}
\begin{proof}
By Corollary \ref{Cor: brick-directed equiv. trim}, $A$ is brick-directed if and only if $\tors A$ is a trim lattice. Moreover, by Proposition \ref{Prop: some properties of brick-directed algebras}, the algebras $A/J$, $eAe$, and $B$ described above are also brick-directed. Hence, the implications follow immediately from Propositions \ref{Prop: brick-directed equiv. left modular} and  \ref{Prop: brick-directed equiv. extremal}.
\end{proof}

As noted earlier, the family of brick-directed algebras properly generalizes the classical family of representation-directed algebras. In fact, from the above lattice theoretical characterizations of brick-directed algebras, we also obtain a useful characterization of representation-directed algebras in terms of their lattice of torsion classes. In particular, we prove Corollary \ref{Cor: Characterization of rep-directed algebras}.

\begin{corollary}\label{Cor: charact. of rep-directed via brick-directed}
Let $A$ be brick-finite algebra. Then, $A$ is representation-directed if and only if $A$ is representation-finite and length of $\tors A$ is $|{\ind A}|$.
\end{corollary}
\begin{proof}
If $A$ is representation-directed, then we have $\ind A=\brick A$. In particular, every indecomposable module in $\modu A$ is rigid, and therefore $A$ is representation-finite (see \cite{Dr} or \cite{MP1}). Moreover, $A$ is evidently brick-directed and therefore, by Proposition \ref{Prop: brick-directed equiv. extremal}, $\tors A$ is extremal, that is, $\tors A$ has a maximal chain of length $|{\brick A}|$. Thus, length of $\tors A$ is obviously $|{\ind A}|$.

To show the converse, first observe that, for any algebra $A$, length of $\tors A$ is always bounded by $|{\brick A}|$, and we obviously have  $|{\brick A}|\leq |{\ind A}|$. Thus, if $A$ is representation-finite and length of $\tors A$ is $|{\ind A}|$, we must have $|{\brick A}|=|{\ind A}|$ and that $\tors A$ admits a maximal chain of length $|{\brick A}|$. These, in particular, imply that $\brick A=\ind A$, and $A$ is brick-directed. By Theorem \ref{Thm: Demonet's correspondence}, this further implies that there is a total order on $\ind A$, implying that $A$ is a representation-directed algebra.
\end{proof}

Regarding the preceding proof, we note that the equality $\ind A=\brick A$ implies that $A$ is representation-finite. However, this does not necessarily imply that $A$ is representation-directed (for more details, see \cite{Dr} and \cite{MP1}).

\begin{remark}
Let us briefly remark on the connection between our results on brick-directed algebras and the notion of (torsional) brick chain filtration, recently introduced in \cite{Ri1}. 
In particular, observe that a brick-finite algebra $A$ is brick-directed if and only if a module in $\modu A$ admits a (torsional) brick chain filtration that contains all elements of $\brick A$.
\end{remark}

As another remarkable property of brick-directed algebras, below we state and prove Corollary \ref{Cor: Uniqueness of dim vector of bricks}. To put this result in perspective, recall that each representation-directed algebra is representation-finite, and over any such algebra every indecomposable is uniquely determined by its dimension vector (see {\cite[IX. Prop. 1.3]{ASS}}). Our next result asserts the analogous property for brick-finite algebras which are brick-directed.

\begin{corollary}\label{Cor: unique dim vector}
Let $A$ be a brick-finite algebra. If $A$ is brick-directed, then each brick in $\modu A$ is uniquely determined by its dimension vector, that is, for two distinct elements $X$ and $Y$ in $\brick A$, we have $[X] \ne [Y]$.
\end{corollary}
\begin{proof}
Since $A$ is brick-directed, we take a sequence $X_1,X_2,\ldots,X_r$ of all bricks over $A$ ordered so that $\Hom_A(X_i,X_j)=0$ for any $i<j$ by Theorem \ref{Thm: Brick-directed & maximal chain of brick-splitting pairs}. As in the proof of Proposition \ref{Prop: walls ordered consistently}, we take $\theta_i \in K_0(\proj A)_\mathbb{Q}$ for each $i$ so that $X_i \in \overline{\mathcal{T}}_{\theta_j}$ if $i \le j$ and $X_i \in \mathcal{F}_{\theta_j}$ if $i>j$.

Now assume that there exist two distinct elements $X$ and $Y$ in $\brick A$ such that $[X]=[Y]$. Without loss of generality, we may write $X=X_l$ and $Y=X_m$ with $l<m$. Then $Y=X_m \in \mathcal{F}_{\theta_l}$ holds, which gives $\theta_l([Y]) < 0$. On the other hand, we also have $X=X_l \in \overline{\mathcal{T}}_{\theta_l}$, which implies $\theta_l([X]) \ge 0$. These contradict each other, because $[X]=[Y]$.
\end{proof}

\begin{remark}
Motivated by the uniqueness property shown in Corollary \ref{Cor: unique dim vector}, and our results from Section \ref{Subsection: Wall-chamber structure of brick-directed algebras} on wall-and-chambers of brick-directed algebras, we pose the following question on brick-infinite algebras:
Let $A$ be brick-directed, and assume that $X$ and $Y$ are bricks with $[X]=[Y]$ in $K_0(\modu A)$. Then does there exist a brick $Z$ such that $[X]=[Y]=[Z]$ and $\Theta_X \subset \Theta_Z$ and $\Theta_Y \subset \Theta_Z$?  
\end{remark}

We note that the uniqueness property considered in Corollary \ref{Cor: unique dim vector} can also be interpreted in terms of paths in a Newton polytope. To explain this point, let us first recall some terminology.

Let $M \in \modu A$. Then, the \emph{Newton polytope} $\mathrm{N}(M)$ of $M$ in $K_0(\modu A)_\mathbb{R}$ is defined by \cite{BKT,Fe1} as the convex hull of the subset $\{[L] \mid \text{$L$ is a submodule of $M$}\}$. For basic properties of Newton polytopes, see also \cite{AH+,Fe2}.

For any two elements $v,w \in K_0(\modu A)_\mathbb{R}$, we write $v \le w$ if $w-v=\sum_{i=1}^n a_i[S_i]$ with $a_i \ge 0$ for each $i$, and $v<w$ means $v \le w$ and $v \ne w$.  We call a sequence $v_0<v_1<\cdots<v_l$ \emph{an increasing path} of the Newton polytope $\mathrm{N}(M)$ if $v_i$ is a vertex of $\mathrm{N}(M)$, for each $i \in \{0,\ldots,l\}$, and the line segment $[v_{i-1},v_i]$ is an edge of $\mathrm{N}(M)$ with $v_{i-1}<v_i$, for each $i \in \{1,\ldots,l\}$. 

Moreover, for any $\mathcal{T} \in \tors A$ and $M \in \modu A$, we define $t_\mathcal{T}(M) \subseteq M$ as the unique submodule $L \subseteq M$ such that $L \in \mathcal{T}$ and $M/L \in \mathcal{T}^\perp$.

In our context, it is enough to consider the case where $M$ is the direct sum of all bricks in $\modu A$ with $A$ brick-finite. Then $\mathrm{N}(M)$ and $\tors A$ are related as follows.

\begin{proposition}\label{Prop: Newton polytopes}\cite[Subsection 5.3]{AH+} \cite[Subsection 9.2]{Fe2}
Let $A$ be brick-finite, and set $M$ as the direct sum of all $X \in \brick A$. 
\begin{enumerate}
\item 
There exists a bijection
$$\tors A \to \{\text{vertices of $\mathrm{N}(M)$}\}$$ 
given by $\mathcal{T} \mapsto [t_\mathcal{T}(M)]$.
\item 
The map in (1) induces a bijection 
$$ \{\text{maximal chains in $\tors A$}\} \to \{\text{increasing paths in $\mathrm{N}(M)$ from $0$ to $[M]$}\},$$
which preserves lengths.
\item
Let $0=\mathcal{T}_0 \subset \mathcal{T}_1 \subset \cdots \subset \mathcal{T}_l=\modu A$ be a maximal chain in $\tors A$, and assume that it is sent to $0=v_0<v_1<\cdots<v_l=[M]$ under the bijection (2). Then, for any $i \in \{1,\ldots,l\}$, we have $v_i-v_{i-1} \in \mathbb{Z}_{\ge 1}[X_i]$ for the label $X_i$ of the arrow $\mathcal{T}_i \to \mathcal{T}_{i-1}$.
\end{enumerate}
\end{proposition}

An increasing path $v_0<v_1<\cdots<v_l$ in $\mathrm{N}(M)$ is said to be \emph{indivisible} if $v_i-v_{i-1}$ is an indivisible element in $K_0(\modu A)$, for each $i \in \{1,\ldots,l\}$. 

\begin{lemma}\label{Lemma: indivisible paths}
Let $A$ be brick-finite, and set $M$ as the direct sum of all $X \in \brick A$. Assume that $0=v_0<v_1<\cdots<v_l=[M]$ is an increasing path in $\mathrm{N}(M)$ from $0$ to $[M]$. Then the path is indivisible if and only if $l=|{\brick A}|$.
\end{lemma}

\begin{proof}
By Proposition \ref{Prop: Newton polytopes}(3), we take a positive integer $m_i \in \mathbb{Z}_{\ge 1}$ such that $v_i-v_{i-1}=m_i[X_i]$ for each $i \in \{1,\ldots,l\}$. Then we get 
$$\sum_{X \in \brick A}[X]=[M]=v_l-v_0=\sum_{i=1}^l m_i[X_i] .$$

We first show the ``if'' part. The bricks $X_1,\ldots,X_l$ are distinct by Proposition \ref{Prop: labels distinct in a path}, so the assumption $l=|{\brick A}|$ implies that $X_1,\ldots,X_l$ are all the bricks in $\modu A$. Thus $\sum_{X \in \brick A}[X]=\sum_{i=1}^l m_i[X_i]$ above gives $m_i=1$, for any $i$. This implies that the increasing path $0=v_0<v_1<\cdots<v_l=[M]$ is indivisible, because Lemma \ref{Lemma: relative-interior of walls}(2) implies that each $[X_i]$ is indivisible.

Next we show the ``only if'' part. Assume that $0=v_0<v_1<\cdots<v_l=[M]$ is indivisible. Then we have $m_i=1$, for any $i$. Thus we have $\sum_{X \in \brick A}[X]=\sum_{i=1}^l [X_i]$. This implies $l=|{\brick A}|$, since $X_1,\ldots,X_l$ are distinct.
\end{proof}

Now we are ready to prove Theorem \ref{Newton polytope thm in Introduction} and give an interesting characterization of brick-directed algebras in terms of properties of the Newton polytope considered above.

\begin{theorem}\label{Thm: Newton polytope}
Let $A$ be brick-finite, and set $M$ as the direct sum of all $X \in \brick A$. Then, $A$ is brick directed if and only if 
there exists an indivisible increasing path in $\mathrm{N}(M)$ from $0$ to $[M]$.
\end{theorem}

\begin{proof}
As in the proof of Proposition \ref{Prop: brick-directed equiv. extremal}, $A$ is brick directed if and only if $\tors A$ has a maximal chain of length $|{\brick A}|$. The latter is equivalent to the existence of an increasing path of $\mathrm{N}(M)$ of length $|{\brick A}|$ from $0$ to $[M]$, by Proposition \ref{Prop: Newton polytopes}(2), and also to the existence of an indivisible increasing path of $\mathrm{N}(M)$ from $0$ to $[M]$, by Lemma \ref{Lemma: indivisible paths}.
\end{proof}

\medskip

We finish this section with a remark and further observations on Corollary \ref{Cor: unique dim vector}, including a question on the behavior of arbitrary brick-directed algebras.

\begin{remark}\label{Rem: Unique dim vect but not brick-direcetd}
To suitably put the important uniqueness property of Corollary \ref{Cor: unique dim vector} in perspective, particularly with regard to some open conjectures on bricks, we separately discuss the case for brick-finite as well as brick-infinite algebras.

\begin{enumerate}
\item Observe that if $A$ is a brick-finite algebra whose bricks are uniquely determined by their dimension vectors, we cannot necessarily conclude that $A$ is brick-directed. For instance, let $A=kQ/I$, where $Q$ is a given by  
$$
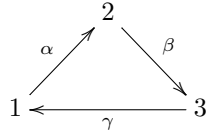
\xymatrix{& 2 \ar[dr]^\beta & \\
1 \ar[ur]^\alpha & & 3 \ar[ll]^\gamma }$$
and $I=\langle \beta\alpha, \gamma\beta, \alpha\gamma\rangle$. Then, $\ind A=\{{\begin{smallmatrix} 1\end{smallmatrix}}, {\begin{smallmatrix} 2\end{smallmatrix}}, {\begin{smallmatrix} 3\end{smallmatrix}}, {\begin{smallmatrix} 1\\2\end{smallmatrix}}, {\begin{smallmatrix} 2\\3\end{smallmatrix}}, {\begin{smallmatrix} 3\\1\end{smallmatrix}}\}
$, 
and $\ind A=\brick A$. Obviously, each element of $\ind A$ is uniquely determined by its dimension vector. However, ${\begin{smallmatrix} 1\\2\end{smallmatrix}} \rightarrow {\begin{smallmatrix} 3\\1\end{smallmatrix}} \rightarrow {\begin{smallmatrix} 2\\3\end{smallmatrix}} \rightarrow {\begin{smallmatrix} 1\\2\end{smallmatrix}}$ is a brick-cycle in $\modu A$. Thus, $A$ is not brick-directed.

\item As observed previously, there exist various families of brick-infinite algebras that are brick-directed (see Prop. \ref{Prop: Hereditary brick-directed algs} and Cor. \ref{Various families of brick-directed algebras}). We note that, in contrast to the uniqueness property in Corollary \ref{Cor: unique dim vector}, one cannot expect a similar statement for the brick-infinite brick-directed algebras. 
More precisely, let $A$ be a brick-infinite algebra over an algebraically closed field. Then, by the so-called ``Second brick-Brauer-Thrall" conjecture (2nd bBT), first posed by the third-named author in 2019, it is expected that $\brick A$ contains an infinite set of bricks with the same dimension vector. For the original motivations of the 2nd bBT, and the most recent state of art on this conjecture, respectively, see \cite{Mo2, Mo3}, \cite{MP2, MP3}, and the references therein. Meanwhile, one can naturally ask if, for an arbitrary brick-directed algebra $A$, and for each dimension vector $\underline{d} \in \mathbb{Z}^n_{\geq 0}$, the module variety $\modu (A,\underline{d})$ always admits at most one brick component (for terminology and some motivations, see \cite{MP1, MP4}).
\end{enumerate}
\end{remark}

\newpage

\textbf{Acknowledgments.} We thank Hugh Thomas for some stimulating discussions about trim lattices at the earlier stage of this project, and for some helpful follow-up comments. We also thank Adrien Segovia for drawing our attention to \cite[Theorem 3.2]{Mu}, and Eric Hanson for pointing out some related studies in \cite[Section 5]{BH}. 
SA was supported by Early-Career Scientist JSPS Kakenhi grant number 23K12957. 
OI was supported by JSPS Grant-in-Aid for Scientific Research (B) 22H01113, (B) 23K22384.
KM was supported by Early-Career Scientist JSPS Kakenhi grant number 24K16908. CP was supported by the Natural Sciences and Engineering Research Council of Canada (RGPIN-2018-04513) and by the Canadian Defence Academy Research Programme.

\end{document}